\documentclass[a4paper,11pt,reqno,noindent]{amsart}
\usepackage[centertags]{amsmath}
\usepackage{amsfonts,amssymb,amsthm,dsfont,cases,amscd,esint,enumerate, stmaryrd}
\usepackage[T1]{fontenc}
\usepackage[english]{babel}
\usepackage[applemac]{inputenc}
\usepackage{newlfont}
\usepackage{color}
\usepackage[body={17cm,21.5cm}, top = 1in, bottom = 1.5in, centering]{geometry} 
\usepackage{fancyhdr}
\usepackage{enumerate}
\usepackage{comment}

\theoremstyle{plain}
\newtheorem{theor}{Theorem}[section]
\newtheorem{lem}[theor]{Lemma}
\newtheorem{prop}[theor]{Proposition}
\newtheorem{cor}[theor]{Corollary}
\theoremstyle{definition}

\newtheorem{rem}[theor]{Remark}
\newtheorem{defin}[theor]{Definition}

\mathchardef\emptyset="001F
\numberwithin{equation}{section}

\newcommand{\Xd}{\mathbb X}

\newcommand{\R}{\mathbb R}

\newcommand{\T}{\mathbb T}

\newcommand{\Pc}{\mathcal{P}}

\newcommand{\Id}{\operatorname{Id}}
\newcommand{\E}{\mathbb{E}}
\newcommand{\Var}{\operatorname{Var}}

\newcommand{\Div}{{\operatorname{div}}}
\newcommand{\Sym}{{\operatorname{sym}}}

\usepackage[colorlinks,citecolor=black,urlcolor=black]{hyperref}
\usepackage{tikz}
\usetikzlibrary{fit}
\usetikzlibrary{shapes.geometric}
\usetikzlibrary{calc}

\usepackage[colorlinks,citecolor=black,urlcolor=black]{hyperref}

\title[Correlation estimates for singular interactions]{Correlation estimates for Brownian particles\\with singular interactions}

\author[M.~Duerinckx]{Mitia Duerinckx}
\address{\sc M. Duerinckx. Universit\'e Libre de Bruxelles, D\'epartement de Ma\-th\'e\-matiques, 1050 Brussels, Belgium}
\email{mitia.duerinckx@ulb.be}
\author[P.-E. Jabin]{Pierre-Emmanuel Jabin}
\address[Pierre-Emmanuel Jabin]{Penn State University, Department of Mathematics, State College, PA 16802, USA}
\email{pejabin@psu.edu}

\begin{document}

\begin{abstract}
We study particle systems with singular pairwise interactions and non-vanishing diffusion in the mean-field scaling. A classical approach to describing corrections to mean-field behavior is through the analysis of correlation functions. For bounded interactions, the optimal estimates on correlations are well known: the $m$-particle correlation function is $G_{N,m}=O(N^{1-m})$ for all $m$. Such estimates, however, have remained out of reach for more singular interactions. In this work, we develop a new framework based on {\it linearized} correlation functions, which allows us to derive robust bounds for systems with merely square-integrable interaction kernels, providing the first systematic control of correlations in the singular setting. Although at first not optimal, our estimates can be partially refined a posteriori using the BBGKY hierarchy: in the case of bounded interactions, our method recovers the known optimal estimates with a simplified argument. As key applications, we establish the validity of the Bogolyubov correction to mean field and prove a central limit theorem for the empirical measure, extending these results beyond the bounded interaction regime for the first time.
\end{abstract}

\maketitle
\setcounter{tocdepth}{1}
\tableofcontents
\allowdisplaybreaks

\section{Introduction}

\subsection{General overview}
We consider the Langevin dynamics for a system of $N$ exchangeable particles with pairwise interactions and non-vanishing diffusion in velocities. Focusing for simplicity on systems on the periodic torus $\T^d$, in space dimension $d\ge1$, the dynamics is given by the following system of SDEs: denoting by $Z_{N,i}=(X_{N,i},V_{N,i})$ the positions and velocities of the particles in phase space $\Xd:=\T^d\times\R^d$,
\begin{equation}\label{eq:traj}
\left\{\begin{array}{ll}
dX_{N,i}=V_{N,i}dt,\\
dV_{N,i}=\frac1N\sum_{j:j\ne i}K(X_{N,i},X_{N,j})dt+\sqrt2dB_t^i,\qquad1\le i\le N,
\end{array}\right.
\end{equation}
where $K:(\T^d)^2\to\R^d$ is an interaction kernel and where $\{B^i\}_{1\le i\le N}$ are independent standard Brownian motions in $\T^d$. The scaling reflects the mean-field regime. Switching to a statistical perspective, we introduce the joint probability density $F_N$ on the $N$-particle phase space $\Xd^N$. The dynamics~\eqref{eq:traj} then formally leads to the Liouville equation
\begin{equation}\label{eq:Liouville}
\partial_tF_N+\sum_{i=1}^Nv_i\cdot\nabla_{x_i}F_N=\sum_{i=1}^N\triangle_{v_i}F_N-\frac1N\sum_{i\ne j}^NK(x_i,x_j)\cdot\nabla_{v_i}F_N.
\end{equation}
We assume the particles are exchangeable, meaning $F_N$ is symmetric with respect to permutations of the variables $z_i=(x_i,v_i)\in\Xd$. For simplicity, we assume that at initial time $t=0$ particles are chaotic in the sense of
\begin{equation}\label{eq:chaotic-data}
F_N|_{t=0}=f_\circ^{\otimes N},
\end{equation}
for some $f_\circ\in \Pc\cap L^1(\Xd)$, but this initial assumption is not really restrictive in view of de Finetti's theorem.
Although we focus on underdamped dynamics on $\T^d$, our methods extend with minor adaptations to particle systems on $\R^d$, with or without confinement, as well as to the corresponding overdamped dynamics (see Remark~\ref{rem:overdamped} below).

In the mean-field limit $N\uparrow\infty$, we consider the evolution of the joint phase-space density of finite subsets of typical particles, as described by marginals of $F_N$,
\[F_{N,m}(z_1,\ldots,z_m)\,:=\,\int_{\Xd^{N-m}}F_N(z_1,\ldots,z_N)\,dz_{m+1}\ldots dz_N,\qquad 1\le m\le N.\]
As is well-known, for chaotic initial data~\eqref{eq:chaotic-data}, particle correlations are expected to remain negligible over time to leading order: more precisely, for all $m\ge1$,
\[F_{N,m}-F_{N,1}^{\otimes m}\,\to\,0,\qquad\text{as $N\uparrow\infty$}.\]
As a consequence of this so-called propagation of chaos, we could deduce the mean-field approximation
\begin{equation}\label{eq:MFlim}
F_{N,m}\,\to\,f^{\otimes m},\qquad\text{as $N\uparrow\infty$},
\end{equation}
where $f$ satisfies the Vlasov-Fokker-Planck equation
\begin{equation}\label{eq:VFP}
\left\{\begin{array}{l}
\partial_tf+v\cdot\nabla_xf-\triangle_vf+(K\ast f)\cdot\nabla_vf=0,\\
f|_{t=0}=f_\circ,
\end{array}\right.
\end{equation}
with the short-hand notation $(K\ast f)(x)=\int_{\T^d\times\R^d}K(x,x')f(x',v')\,dx'dv'$. Such a result has been established in various settings and we refer e.g.\@ to recent work~\cite{BJS-22} and references therein.

To capture deviations from the mean-field limit~\eqref{eq:MFlim}, a classical approach is to study the associated correlation functions $\{G_{N,m}\}_{1\le m\le N}$, or so-called Ursell functions, which describe connected correlations between the particles. These considerations have their origins in equilibrium statistical mechanics, see in particular~\cite[Chapter~13]{Mayer-40} and~\cite[Chapter~4]{Ruelle-69}, and play an important role in some recent developments in kinetic theory, e.g.~\cite{PS-17,BGS-17,DSR-21,DJ-24}.
More precisely, correlation functions associated with the joint density $F_N$ are defined via the cluster expansion
\begin{equation}\label{eq:cluster-exp}
F_N(z_1,\ldots,z_N)=\sum_{\pi\vdash{}[ N]}\prod_{B\in\pi}G_{N,\sharp B}(z_B),
\end{equation}
where the sum is over partitions $\pi$ of the index set $[N]=\{1,\ldots,N\}$, where the product is over blocks of the partition,~$B\in\pi$, and where we use the short-hand notation $z_B=(z_{i_1},\ldots,z_{i_k})$ for $B=\{i_1,\ldots,i_k\}$. We also denote by $\sharp B$ the cardinality of the block $B$. Together with the `maximality' requirement
\[\int_{\Xd}G_{N,m}(z_1,\ldots,z_m)\,dz_\ell=0,\qquad\text{for $1\le \ell\le m$ and $m\ge2$,}\]
it is easily checked that this expansion~\eqref{eq:cluster-exp} uniquely defines correlation functions $\{G_{N,m}\}_{1\le m\le N}$.
Alternatively, via M\"obius inversion formula over the lattice of set partitions, correlation functions can be recursively defined by
\begin{equation}\label{eq:correl}
G_{N,m}(z_1,\ldots,z_m)=\sum_{\pi\vdash{}[ m]}(\sharp\pi-1)!(-1)^{\sharp\pi-1}\prod_{B\in\pi}F_{N,\sharp B}(z_B),\qquad1\le m\le N,
\end{equation}
where $\sharp\pi$ stands for the number of blocks in a partition $\pi$.
This means in particular
\begin{gather*}
G_{N,1}=F_{N,1},\qquad G_{N,2}=F_{N,2}-F_{N,1}^{\otimes2},\qquad G_{N,3}=F_{N,3}-3\,\Sym(F_{N,1}\otimes F_{N,2})+2F_{N,1}^{\otimes3},
\end{gather*}
and so on, where `$\Sym$' stands for symmetrization of tensor fields.
Formal considerations based on the BBGKY hierarchy suggest that these correlations should scale like
\begin{equation}\label{eq:scaling-GNm}
G_{N,m}=O(N^{1-m}).
\end{equation}
Such estimates would yield a refined understanding of the statistical deviations from the mean-field limit, but their derivation is challenging. So far, the bounds~\eqref{eq:scaling-GNm} have been rigorously established only for systems with smooth interactions~\cite{D-21} and for bounded interactions with non-vanishing diffusion~\cite{Hess_Childs_2023}. In this work, we aim to extend these results to systems with more singular interactions, covering the second-order setting~\eqref{eq:traj} where the diffusion is degenerate and acts only in velocity.

Handling unbounded interactions presents significant challenges, particularly for second-order dynamics. For instance, one cannot simply assume $K(x,\cdot)\in L^p(\Xd)^d$ since $K$ depends only on positions while velocities range over the unbounded space $\R^d$. As in~\cite{BJS-22}, this forces the introduction of suitable weighted spaces to avoid mixed-norm complications. At the same time, the diffusion in velocity remains crucial to estimate certain terms in the hierarchy.

Although we cannot reach the expected $O(N^{1-m})$ scaling in the singular setting, we show that suboptimal correlation bounds can still be derived and effectively used to capture corrections to the mean-field behavior. Our approach relies on hierarchical methods, with one key twist: since the BBGKY hierarchy for standard correlation functions becomes too singular in case of unbounded interactions, we develop a new framework of {\it linearized} correlation functions, inspired by~\cite{BDJ-24,BDM_25}, for which linear hierarchical methods remain tractable. This allows us to obtain direct a priori estimates that, although initially suboptimal, can be sharpened a posteriori by exploiting the BBGKY hierarchy. This a posteriori refinement is reminiscent of a similar procedure in~\cite{BDM_25}. In case of bounded interactions, our approach recovers the optimal estimates of~\cite{Hess_Childs_2023}, but with a significantly simpler argument.

\subsection{Main results}

The following result provides nontrivial $L^2$ estimates on correlation functions in case of square-integrable interaction forces deriving from a potential with exponential integrability. More precisely, we consider weighted $L^2$ norms with inverse Maxwellian weight
\[\omega_\beta(z)\,:=\,e^{\frac12\beta|v|^2},\qquad z=(x,v).\]
Our estimates do not match the expected scaling~\eqref{eq:scaling-GNm}: we only show $G_{N,m}=O(N^{-m/2})$ in $L^2$, cf.~\eqref{eq:L2bnd} below, together with the weak convergence $N^{m/2}G_{N,m}\rightharpoonup0$, cf.~\eqref{eq:wL2bnd}. We do not know whether stronger estimates can be expected to hold in the present singular setting. As we shall see, the present estimates are anyway already sufficient for various applications.

\begin{theor}[Correlation estimates]\label{th:main}
Assume that $K$ belongs to $L^2((\T^d)^2)^d$ and derives from a potential,
\[K(x,y)=-\nabla W(x-y),\]
and further assume that the latter satisfies, for some $\beta>0$,
\[W_-\in L^\infty(\T^d),\qquad \sup_{x\in\T^d}\int_{\T^d}(1+|K(x,y)|^2)\,e^{\beta W(y)}\,dy<\infty,\]
where $W_-$ stands for the negative part of $W$.
Assume that there is some $T_0>0$ and a weak solution $f$ of~\eqref{eq:VFP} on $[0,T_0]$ with, for some $\beta>0$,
\begin{eqnarray*}
f&\in& L^\infty_t(0,T_0;L^\infty_x(\T^d;L^2_v(\omega_\beta))),\\
W\ast f&\in& W^{1,\infty}_t(0,T_0;L^\infty_x(\T^d))\cap L^\infty_t(0,T_0;W^{1,\infty}_x(\T^d)).
\end{eqnarray*}
Then there exist a time $T_*\in(0,T_0]$ and some $\beta_*>0$ such that for all $2\le m\le N$ and $t\in[0,T_*]$,
\begin{equation}\label{eq:L2bnd}
\Big(\int_{\Xd^m}|G_{N,m}(t)|^2\omega_{\beta_*}^{\otimes m}\Big)^\frac12\,\le\, C_mN^{-\frac m2},
\end{equation}
for some constant $C_m$ only depending on $d,m,K,f$.
In addition, in the weak sense, for all $m\ge3$,
\begin{equation}\label{eq:wL2bnd}
N^{\frac m2}G_{N,m}\overset*\rightharpoonup0,\qquad\text{in $L^\infty(0,T_*;L^2(\omega_{\beta_*}^{\otimes m}))$}.
\end{equation}
\end{theor}

\begin{rem}[Overdamped dynamics]\label{rem:overdamped}
For the corresponding overdamped particle system (first-order dynamics), the same correlation estimates hold on $L^2(\T^{dm})$ if we only assume $K\in L^2((\T^d)^2)^d$ with $(\Div(K))_-\in L^\infty((\T^d)^2)$, without requiring $K$ to derive from a potential.
\end{rem}

As a first application of the above correlation estimates, we deduce the (qualitative) accuracy of the Bogolyubov correction to mean-field approximation. Note that propagation of chaos in item~(i) below can be proven for more singular interactions, cf.~\cite{BJS-22}, but this is the first result regarding the Bogolyubov correction beyond the case of bounded interactions~\cite{D-21,Hess_Childs_2023}.

\begin{cor}[Bogolyubov correction]\label{cor:Bogo}
Let the same assumptions hold as in Theorem~\ref{th:main}.
Next to the solution $f$ of the mean-field equation~\eqref{eq:VFP}, we consider the weak solution $f_N$ of the corrected equation
\begin{equation*}
\left\{\begin{array}{l}
\partial_tf_N+v\cdot\nabla_xf_N-\triangle_vf_N+(K\ast f_N)\cdot\nabla_vf_N
=
-\tfrac1N\int_\Xd K(x-x_*)\cdot\nabla_{v}(\bar G_2-f^{\otimes2})(z,z_*)\,dz_*,\\
f_N|_{t=0}=f_\circ,
\end{array}\right.
\end{equation*}
where $\bar G_2$ satisfies the Bogolyubov equation
\begin{equation}\label{eq:Bogo}
\left\{\begin{array}{l}
\partial_t \bar G_2+L_f^{(2)}\bar G_2
\,=\,
-\sum_{k\ne\ell}^2\Big(K(x_k-x_\ell)-(K\ast f)(x_k)\Big)\cdot\nabla_{v_k}f^{\otimes2},\\
g|_{t=0}=0.
\end{array}\right.
\end{equation}
where $L_f^{(2)}$ is the $2$-particle Vlasov operator linearized at $f$, on $L^2(\omega_\beta^{\otimes2})$,
\[L_f^{(2)}=L_f\otimes\Id+\Id\otimes L_f,\qquad L_fh=v\cdot\nabla_{x}h-\triangle_{v}h+(K\ast f)\cdot\nabla_{v}h+(K\ast h)(x)\cdot\nabla_{v}f.\]
Then the following results hold:
\begin{enumerate}[(i)]
\item \emph{Propagation of chaos:} for all $1\le m\le N$ and $t\in[0,T_*]$,
\[\|F_{N,m}(t)-f(t)^{\otimes m}\|_{L^2(\omega_{\beta_*}^{\otimes m})}\le C_mN^{-1},\]
for some constant $C_m$ only depending on $d,m,K,f$.
\smallskip\item \emph{Bogolyubov correction:}
\begin{equation*}\begin{array}{rlll}
N(F_{N,1}-f_N)&\overset*\rightharpoonup& 0,\qquad&\text{in $L^\infty(0,T_*;L^2(\omega_{\beta_*}))$,}\\
NG_{N,2}&\overset*\rightharpoonup&\bar G_2,\qquad&\text{in $L^\infty(0,T_*;L^2(\omega_{\beta_*}^{\otimes2}))$,}
\end{array}\end{equation*}
and in addition, for all $m\ge1$,
\[\qquad N(F_{N,m}-f_N^{\otimes m})\,
\overset*\rightharpoonup\,\sum_{1\le k<\ell\le m}\bar G_2(z_k,z_\ell)f^{\otimes m-2}(z_{[m]\setminus\{k,\ell\}}),\quad\text{in $L^\infty(0,T_*;L^2(\omega_{\beta_*}^{\otimes m}))$.}\]
\end{enumerate}
\end{cor}

As a second application of the above correlation estimates, we establish a (qualitative) central limit theorem (CLT) for fluctuations of the empirical measure,
\[\mu_N\,:=\,\textstyle\frac1N\sum_{i=1}^N\delta_{Z_{N,i}},\]
where we recall that $t\mapsto\{Z_{N,i}(t)\}_{1\le i\le N}$ stands for particle trajectories~\eqref{eq:traj}, which are well-posed in the strong sense if we further assume e.g.\@ $K\in W^{1,1}$, cf.~\cite{Champagnat-Jabin,LeBris-Lions-19}. Due to their subcritical scaling, correlation estimates~\eqref{eq:L2bnd} are of no use on their own to deduce a CLT, but we can rely on the weak convergence~\eqref{eq:wL2bnd}.
To the best of our knowledge, this is the first CLT for the empirical measure beyond the case of bounded interactions.
As the proof relies on hierarchical techniques, it only captures the fluctuations of the empirical measure at fixed times, without providing information about time correlations along trajectories.
In the case of bounded interactions, a functional CLT for the law of the full process $\{N^{1/2}(\mu_N-f)(t)\}_{t\ge0}$ was established in~\cite{Wang-Zhao-Zhu-23}, and optimal error estimates were obtained in~\cite{D-21,BD_2024} for smooth interactions.

\begin{cor}[CLT]\label{cor:CLT}
Let the same assumptions hold as in Theorem~\ref{th:main}, and further assume that~$K$ belongs to $W^{1,1}((\T^d)^2)^d$ to ensure strong well-posedness of the trajectories. Then we have for all $t\in[0,T_*]$ and $\varphi\in C^\infty_c(\Xd)$,
\begin{equation}\label{eq:CLT-conv}
N^\frac12\int_\Xd\varphi(\mu_N-f)(t)\,\to\,\int_\Xd\varphi\nu(t),\qquad\text{in law},
\end{equation}
where the limit fluctuation $\nu$ is the centered Gaussian process that is the unique almost sure distributional solution of the Gaussian linearized Dean-Kawasaki SPDE,
\[\left\{\begin{array}{l}
\partial_t\nu+v\cdot\nabla_x\nu-\triangle_v\nu+(K\ast\nu)\cdot\nabla_vf+(K\ast f)\cdot\nabla_v\nu=\Div_v(\sqrt f\xi),\\
\nu|_{t=0}=\nu^\circ,
\end{array}\right.\]
where $\xi$ is a vector-valued space-time white noise on $\R^+\times\T^d\times\R^d$ and where $\nu^\circ$ is the Gaussian field describing the fluctuations of the initial empirical measure in the sense that $N^{1/2}\int_\Xd\varphi(\mu_N-f)|_{t=0}$ converges in law to $\int_\Xd\varphi\nu^\circ$ for all $\varphi\in C^\infty_c(\Xd)$.
\end{cor}

\subsection{Improved estimates for $K\in L^\infty$}
In case of bounded interaction forces, we show that all our estimates can be improved: we obtain the optimal scaling~\eqref{eq:scaling-GNm} and the estimates are shown to hold globally in time (up to exponential time growth). In this way, we recover the result previously proven in~\cite{Hess_Childs_2023}, but with a substantially shorter proof.
Note that we do not try to improve on the $m$-dependence in these estimates.
We refer to Section~\ref{sec:Linfty} for details.

\begin{theor}\label{th:main-Linfty}
Let $K\in L^\infty((\T^d)^2)^d$, $f_\circ\in L^2(\omega_\beta)$ for some $\beta>0$, and let $f$ be a global weak solution of~\eqref{eq:VFP}.
Then the following improved correlation estimates hold for all $1\le m\le N$ and $t\ge0$,
\begin{equation}\label{eq:Linftybnd}
\Big(\int_{\Xd^m}|G_{N,m}(t)|^2\omega_{\beta(t)}^{\otimes m}\Big)^\frac12\,\le\, (Ce^{Ct})^{C^m}N^{1-m},
\end{equation}
where we have set $\beta(t):=\frac{\beta}{1+4\beta t}$ and where the constant $C$ only depends on $d,\beta,K,f_\circ$.
\end{theor}

\begin{rem}
From the above improved correlation estimates, the statement of Corollary~\ref{cor:Bogo}(ii) on the accuracy of the Bogolyubov correction can also be improved accordingly: instead of weak convergences, we can deduce optimal global error estimates $O(e^{Ct}N^{-1})$ in $L^2(\omega_{\beta(t)})$ for all~$t\ge0$. We skip the details for shortness.
\end{rem}

\subsection*{Notation}
\begin{enumerate}[---]
\item We denote by $C\ge1$ any constant that only depends on the space dimension $d$ and on controlled norms of the interaction kernel $K$ and of the mean-field solution $f$. We use the notation $\lesssim$ for $\le C\times$ up to such a multiplicative constant $C$. We add subscripts to $C,\lesssim$ to indicate dependence on other parameters.
\smallskip\item For $\beta>0$, we denote by $\omega_\beta(z)=e^{\beta|v|^2/2}$ the inverse Maxwellian weight for $z=(x,v)\in\Xd$.
\smallskip\item We use the short-hand notation $[m]=\{1,\ldots,m\}$, and we set $z_B=(z_{i_1},\ldots,z_{i_k})$ for an index subset $B=\{i_1,\ldots,i_k\}$.
\end{enumerate}

\section{Tools for global hierarchical estimates}
To solve hierarchies of differential inequalities, as those appearing in the analysis of the BBGKY hierarchy, standard methods are based on Cauchy-Kovalevski type arguments and lead to estimates that are restricted to short times (see e.g.~\cite[Section~1.11.2]{Golse-rev2}). The following result shows that such estimates can be extended globally in time whenever exponential a priori estimates are already available. This follows from similar computations in~\cite[Section~3]{PPS-19} and will be repeatedly used in the sequel.

\begin{lem}[Global hierarchical estimates]\label{lem:hier}
Given $T\in(0,\infty]$, let $\{a_n\}_{n\ge1}$ be a sequence of continuous maps $a_n:[0,T)\to\R^+$ satisfying the following hierarchy of differential inequalities for some parameters~$A,R\ge1$: for all $n\ge1$,
\begin{equation}\label{eq:ineq-an}
\left\{\begin{array}{ll}
\tfrac{d}{dt}a_n\,\le\,
n(a_n+a_{n+1})
+n^3R^{-2}(a_{n-1}+a_{n-2}),&\text{on $[0,T)$},\\[1mm]
a_n(0)\le n^{2n}(AR^{-1} )^n.&
\end{array}\right.
\end{equation}
Further assume the following a priori estimates for some $B\ge1$: for all $n\ge1$ and $t\in[0,T)$,
\begin{equation}\label{eq:apriori-an}
a_n(t)\le B^n.
\end{equation}
Then there is a constant $C\ge1$ (only depending on $B$, and not on~$T,A,R$) such that we have for all~$n\ge1$ and~$t\in[0,T)$,
\begin{equation}\label{eq:todo-an}
a_n(t)\le e^{Cnt}(Cn^2)^n(AR^{-1})^n.
\end{equation}
\end{lem}

\begin{proof}
We shall show that there is a universal constant $C\ge1$ such that
\begin{equation}\label{eq:todoiter}
a_n(t)\le (CBn^2)^n(AR^{-1} )^n,\qquad\text{for all $0\le t< (CB)^{-1}\wedge T$ and $n\ge1$},
\end{equation}
from which the conclusion~\eqref{eq:todo-an} indeed follows by a direct iteration.
In order to prove~\eqref{eq:todoiter}, we note that for $n\ge\sqrt R$ there is nothing to prove in view of the a priori estimates~\eqref{eq:apriori-an}: indeed,
\[a_n\le1\le n^{2n}R^{-n},\qquad\text{for $n\ge \sqrt R$}.\]
Hence, we only need to prove~\eqref{eq:todoiter} for $n<\sqrt R$.
In addition, without loss of generality, we can assume~$R\ge4$ (say).
In order to estimate $a_n$, we proceed by integrating in time the hierarchy of differential inequalities~\eqref{eq:ineq-an}, and by iterating the resulting integral inequalities, only stopping the iteration when we reach some $a_m$ with $m\ge R$ (not~$\sqrt R$ here!).
More precisely, this leads us to the following truncated Dyson-type expansion,
\begin{eqnarray*}
a_n(t)&\le&\sum_{k=0}^\infty\sum_{j_1,\ldots,j_k}1_{n_0,\ldots,n_k<R}\Big(\prod_{i=1}^km_i\Big)\,a_{n_k}(0)\\
&&\qquad\times\Big(\int_{0\le t_k\le\ldots\le t_1\le t}e^{n_0(t-t_1)}\ldots e^{n_{k-1}(t_{k-1}-t_k)}e^{n_kt_k}dt_1\ldots dt_k\Big)\\
&+&\sum_{k=0}^\infty\sum_{j_1,\ldots,j_k}1_{n_0,\ldots,n_{k-1}<R\le n_k}\Big(\prod_{i=1}^km_i\Big)\\
&&\qquad\times\Big(\int_{0\le t_k\le\ldots\le t_1\le t}e^{n_0(t-t_1)}\ldots e^{n_{k-1}(t_{k-1}-t_k)}a_{n_k}(t_k)\,dt_1\ldots dt_k\Big),
\end{eqnarray*}
where the sum over $j_1,\ldots,j_k$ runs over $\{1,-1,-2\}$, where we set for abbreviation $n_i:=n+\sum_{l=1}^ij_l$, and where we define $m_i=n_{i-1}$ if $j_i=1$ and $m_i=n_{i-1}^3R^{-2}$ if $j_i\in\{-1,-2\}$.
As $n_i\le n+k$ for $0\le i\le k$, we find
\[e^{n_0(t-t_1)}\ldots e^{n_{k-1}(t_{k-1}-t_k)}e^{n_kt_k}\,\le\, e^{(n+k)t}.\]
Inserting this bound into the above, evaluating the remaining time integrals, and using the a priori estimate~\eqref{eq:apriori-an} in the form $a_{n_k}\le B^{n_k}\le B^{n+k}$ (as $n_k\le n+k$), we obtain
\begin{eqnarray}
a_n(t)&\le&e^{nt}\sum_{k=0}^\infty \tfrac{t^ke^{kt}}{k!}\sum_{j_1,\ldots,j_k}1_{n_0,\ldots,n_k<R}\Big(\prod_{i=1}^km_i\Big)\,a_{n_k}(0)\nonumber\\
&+&e^{nt}B^n\sum_{k=0}^\infty \tfrac{(Bt)^ke^{kt}}{k!}\sum_{j_1,\ldots,j_k}1_{n_0,\ldots,n_{k-1}<R\le n_k}\Big(\prod_{i=1}^km_i\Big).\label{eq:estim-an-00}
\end{eqnarray}
We start by examining the second term, which is easier to control. Provided $n_0,\ldots,n_{k-1}<R$, the definition of $m_i$ yields $m_i< n_{i-1}\le n+k$ for all $i$, hence 
\begin{equation}\label{eq:estim-(n+2k)k}
\prod_{i=1}^km_i\,\le\,(n+k)^k\,=\,k^k(1+\tfrac{n}{k})^k\,\le\,k!C^{n+k}.
\end{equation}
For $n<\sqrt R$, recalling the choice $R\ge4$, the condition $n+k\ge n_k\ge R$ implies
\[k\ge R-n\ge \sqrt R(\sqrt R-1)\ge\tfrac12R.\]
Using this observation and~\eqref{eq:estim-(n+2k)k}, the second term in~\eqref{eq:estim-an-00} can be bounded as follows, provided $n<\sqrt R$ and~$CBte^{t}\le e^{-1}$,
\begin{multline}\label{eq:red-an-hier}
e^{nt}B^n\sum_{k=0}^\infty \tfrac{(Bt)^ke^{kt}}{k!}\sum_{j_1,\ldots,j_k}1_{n_0,\ldots,n_{k-1}<R\le n_k}\Big(\prod_{i=1}^km_i\Big)\\
\,\le\,e^{nt}(CB)^n\sum_{k=0}^\infty (CBte^t)^k1_{k\ge\frac12R}
\,\le\,e^{nt}(CB)^ne^{-\frac12R}\,\le\,e^{nt}(CBn)^nR^{-n},
\end{multline}
where in the last inequality we have used the bound $e^{-a}\le n!a^{-n}\le (Cna^{-1})^n$ for $a\ge0$.
It remains to estimate the first term in~\eqref{eq:estim-an-00}, for which we distinguish between three cases.

\medskip\noindent
{\bf Case~1:} $n_k\le n$.\\
In this case, we can find $i_0\in\{0,\ldots,k\}$ such that $n_{i_0}=n$ and $n_i\le n$ for all $i_0<i\le k$. Provided $n_0,\ldots,n_k<R$, recalling that $m_i< n_{i-1}\le n+k$ for all $i$, noting that $m_i\le n^3R^{-2}$ for $i_0<i\le k$ if~$j_i\in\{-1,-2\}$, and noting that there must be at least $\frac12(n-n_k)$ steps with $j_i\in\{-1,-2\}$ for $i_0<i\le k$, we can bound
\[\prod_{i=1}^{i_0}m_i\,\le\,(n+k)^{i_0},\qquad\prod_{i=i_0+1}^{k}m_i\,\le\,(n^3R^{-2})^{\frac12(n-n_k)}n^{k-i_0-\frac12(n-n_k)}\,=\,(nR^{-1})^{n-n_k}\,n^{k-i_0},\]
hence, using again $(n+k)^k\le k!C^{n+k}$, cf.~\eqref{eq:estim-(n+2k)k},
\[\prod_{i=1}^{k}m_i\,\le\,(n+k)^{k}(nR^{-1})^{n-n_k}\,\le\,k!C^{n+k}(nR^{-1})^{n-n_k}.\]
Let us now use this to estimate the first right-hand side term in~\eqref{eq:estim-an-00} in the case $n_k\le n$. Using the initial assumption in form of
\[a_{n_k}(0)\le (n_k)^{2n_k}(AR^{-1})^{n_k}\le n^nA^n(nR^{-1})^{n_k},\]
we get for $Cte^{t}\le e^{-1}$,
\begin{multline*}
e^{nt}\sum_{k=0}^\infty \tfrac{t^ke^{kt}}{k!}\sum_{j_1,\ldots,j_k}1_{n_0,\ldots,n_k<R}\,1_{n_k\le n}\Big(\prod_{i=1}^km_i\Big)\,a_{n_k}(0)\\[-2mm]
\,\le\,e^{nt}(Cn^2)^n(AR^{-1})^{n}\sum_{k=0}^\infty (Cte^{t})^k
\,\le\,e^{nt}(Cn^2)^n(AR^{-1})^{n}.
\end{multline*}

\noindent
{\bf Case~2:} $n<n_k\le \sqrt{A^{-1}R}$.\\
Using~\eqref{eq:estim-(n+2k)k} again, for the first right-hand side term in~\eqref{eq:estim-an-00} restricted to the present case, we find
\begin{multline}\label{eq:estim-rhs-nk>n}
e^{nt}\sum_{k=0}^\infty \tfrac{t^ke^{kt}}{k!}\sum_{j_1,\ldots,j_k}1_{n_0,\ldots,n_k<R}\,1_{n<n_k\le\sqrt{A^{-1}R}}\,\Big(\prod_{i=1}^km_i\Big)\,a_{n_k}(0)\\
\,\le\,C^ne^{nt}\sum_{k=0}^\infty (Cte^{t})^k\sum_{j_1,\ldots,j_k}1_{n<n_k\le\sqrt{A^{-1}R}}\,a_{n_k}(0).
\end{multline}
Using the initial assumption in form of
\[a_{n_k}(0)\,\le\,(n_k)^{2n_k}(AR^{-1})^{n_k}
\,=\,n^{2n}\big(\tfrac{n_k}{n}\big)^{2n}(n_k^2AR^{-1})^{n_k-n}(AR^{-1})^{n},\]
and noting that the bound $n_k\le n+k$ entails
\[\big(\tfrac{n_k}n\big)^{2n}\,\le\,(1+\tfrac kn)^{2n}\,\le\,e^{2k},\]
we obtain, for $n< n_k\le\sqrt{A^{-1}R}$,
\[a_{n_k}(0)\,\le\,C^kn^{2n}(AR^{-1})^{n}.\]
Inserting this into~\eqref{eq:estim-rhs-nk>n}, we thus get for $Cte^t\le e^{-1}$,
\begin{multline*}
e^{nt}\sum_{k=0}^\infty \tfrac{t^ke^{kt}}{k!}\sum_{j_1,\ldots,j_k}1_{n_0,\ldots,n_k<R}\,1_{n<n_k\le\sqrt{A^{-1}R}}\,\Big(\prod_{i=1}^km_i\Big)\,a_{n_k}(0)\\[-2mm]
\,\le\,e^{nt}(Cn^2)^n(AR^{-1})^n\sum_{k=0}^\infty (Cte^{t})^k\,\le\,e^{nt}(Cn^2)^n(AR^{-1})^n.
\end{multline*}

\medskip\noindent
{\bf Case~3:} $n_k>\sqrt{A^{-1}R}$.\\
Using~\eqref{eq:estim-(n+2k)k} again, as well as the a priori bounds~\eqref{eq:apriori-an} at $t=0$ in form of $a_{n_k}(0)\le B^{n_k}\le B^{n+k}$ (as~$n_k\le n+k$), we get
\begin{multline}
e^{nt}\sum_{k=0}^\infty \tfrac{t^ke^{kt}}{k!}\sum_{j_1,\ldots,j_k}1_{n_0,\ldots,n_k<R}\,1_{n_k>\sqrt{A^{-1}R}}\Big(\prod_{i=1}^km_i\Big)\,a_{n_k}(0)\\[-3mm]
\,\le\,e^{nt}(CB)^n\sum_{k=0}^\infty (CBte^{t})^k\sum_{j_1,\ldots,j_k}1_{n_k>\sqrt{A^{-1}R}}.
\end{multline}
In order to estimate the last sum, we distinguish between two further cases, depending on the size of~$n$. If $n\le\frac12\sqrt{A^{-1}R}$, the condition $n+k\ge n_k> \sqrt{A^{-1}R}$ implies $k> \frac12\sqrt{A^{-1}R}$, and thus for~$CBte^{t}\le e^{-1}$,
\begin{multline}
1_{n\le \frac12\sqrt{A^{-1}R}}\,e^{nt}\sum_{k=0}^\infty \tfrac{t^ke^{kt}}{k!} \sum_{j_1,\ldots,j_k} 1_{n_0,\ldots,n_k<R}\,1_{n_k> \sqrt{A^{-1}R}}\Big(\prod_{i=1}^km_i\Big)\,a_{n_k}(0)\\
\,\le\,e^{nt}(CB)^n\sum_{k>\frac12\sqrt{A^{-1}R}}(CBte^t)^k
\,\le\,e^{nt}(CB)^ne^{-\frac12\sqrt{A^{-1}R}}
\,\le\,e^{nt}(CBn^2)^n(AR^{-1})^n.
\end{multline}
If instead $n>\frac12\sqrt{A^{-1}R}$, we can bound for $CBte^t\le e^{-1}$,
\begin{multline}
1_{n>\frac12\sqrt{A^{-1}R}}\,e^{nt}\sum_{k=0}^\infty \tfrac{t^ke^{kt}}{k!}\sum_{j_1,\ldots,j_k}1_{n_0,\ldots,n_k<R}\,1_{n_k> \sqrt{A^{-1}R}}\Big(\prod_{i=1}^km_i\Big)\,a_{n_k}(0)\\
\,\le\,1_{n>\frac12\sqrt{A^{-1}R}}\,e^{nt}(CB)^n\sum_{k=0}^\infty (CBte^{t})^k
\,\le\,1_{n> \frac12\sqrt{A^{-1}R}}\,e^{nt}(CB)^n
\,\le\,e^{nt}(CBn^2)^n(AR^{-1})^{n}.
\end{multline}
Combining the above different cases to estimate the first right-hand side term in~\eqref{eq:estim-an-00}, and combining with the bound~\eqref{eq:red-an-hier} on the second term, the conclusion~\eqref{eq:todoiter} follows for $n<\sqrt R$.
\end{proof}

\section{New framework for correlation estimates}

\subsection{Linear correlations}
The key of our argument for correlation estimates is to start by considering the following {\it linear correlations}, which are obtained by linearizing the correlation functions~$\{G_{N,m}\}_{1\le m\le N}$ around the mean-field approximation $f^{\otimes N}$, where we recall that $f$ stands for the mean-field solution~\eqref{eq:VFP}. More precisely, we define
\begin{equation*}
H_{N,0}\,=\,1,\qquad
H_{N,1}\,=\,F_{N,1}-f,\qquad
H_{N,2}\,=\,F_{N,2}-F_{N,1}\otimes f-f\otimes F_{N,1}+f\otimes f,
\end{equation*}
and more generally, for all $0\le m\le N$,
\begin{equation}\label{eq:defin-HNm}
H_{N,m}\,:=\,\sum_{k=0}^m(-1)^{m-k}\sum_{\sigma\in P^m_k}F_{N,k}(z_\sigma)f^{\otimes m-k}(z_{[m]\setminus\sigma}),
\end{equation}
where $P^m_k$ stands for the set of all subsets of $[m]$ with $k$ elements.
These quantities allow to reconstruct marginals via the following {\it linear} cluster expansion (compare to~\eqref{eq:cluster-exp}),
\begin{equation}\label{eq:cluster-FNmlin}
F_{N,m}=\sum_{k=0}^m\sum_{\sigma\in P^m_k}H_{N,k}(z_\sigma)f^{\otimes m-k}(z_{[m]\setminus\sigma}),\qquad1\le m\le N.
\end{equation}

\begin{rem}
The above linear correlations are formally similar to the dual correlations studied in~\cite{BDJ-24}, and indeed we have the following duality relation: denoting by $\{C_{N,m}\}_m$ the dual correlations associated to a dual Liouville solution $\Phi_N$ as defined in~\cite[Section~2]{BDJ-24}, there holds
\[
\sum_{m=0}^N \binom{N}{m}\int_{\Xd^m} C_{N,m}(t)\,H_{n,m}(t)=\sum_{m=0}^N \binom{N}{m}\int_{\Xd^m} C_{N,m}(0)\,H_{n,m}(0).
\]
Also note that the definition of linear correlations yields for any $\varphi\in C^\infty_c(\Xd)$,
\[\int_{\Xd^m}\varphi^{\otimes m}H_{N,m}=\int_{\Xd^m}\Big(\varphi-\int_{\Xd}\varphi f\Big)^{\otimes m}F_{N,m},\]
thus drawing the link to moments of $\mu_N-f$ as studied e.g.\@ in~\cite{BDM_25}.
Finally, we can compare the above linear correlations to those considered close to equilibrium in~\cite{BGS-17,DSR-21,DJ-24}, except that linearization is taken here at the mean-field approximation $f^{\otimes N}$ away from equilibrium.
\end{rem}

\subsection{BBGKY hierarchies}
Recall the standard BBGKY hierarchy of equations satisfied by marginals $\{F_{N,m}\}_{1\le m\le N}$, which is easily obtained by integrating the Liouville equation~\eqref{eq:Liouville}: for $1\le m\le N$,
\begin{multline}\label{eq:BBGKY0}
\partial_tF_{N,m}+\sum_{i=1}^mv_i\cdot\nabla_{x_i}F_{N,m}-\sum_{i=1}^m\triangle_{v_i}F_{N,m}\\
=-\frac1N\sum_{i\ne j}^mK(x_i,x_j)\cdot\nabla_{v_i}F_{N,m}-\frac{N-m}N\sum_{i=1}^m\int_\Xd K(x_i,x_*)\cdot\nabla_{v_i}F_{N,m+1}(\cdot,z_*)\,dz_*,
\end{multline}
where we set $F_{N,m}\equiv0$ for $m>N$.
From this, we can deduce a corresponding hierarchy of equations for linear correlations $\{H_{N,m}\}_{1\le m\le N}$.
As we shall see, this hierarchy is much more tractable than the hierarchy for nonlinear correlations $\{G_{N,m}\}_{1\le m\le N}$ (see Lemma~\ref{lem:hier-corr} below): in particular, it will allow us to obtain estimates for merely $L^2$ interaction forces, which seems impossible from the hierarchy for nonlinear correlations.

\begin{lem}[Hierarchy for linear correlations]\label{lem:BBGKY-HNm}
The above-defined linear correlations $\{H_{N,m}\}_m$ satisfy for all $0\le m\le N$,
\begin{multline*}
\partial_tH_{N,m}+L_{N,m}H_{N,m}=
\frac{1}N\sum_{i\ne j}^mf(z_j)\int_\Xd K(x_i,x_*)\cdot\nabla_{v_i}H_{N,m}(z_{[m]\setminus\{j\}},z_*)\,dz_*\\
-\frac{N-m}N\sum_{i=1}^m(\nabla_{v} f)(z_i)\cdot\int_\Xd K(x_i,x_*)H_{N,m}(z_{[m]\setminus\{i\}},z_*)\,dz_*\\
-\frac1N\sum_{i\ne j}^mf(z_j)(\nabla_vf)(z_i)\cdot\Big(K(x_i,x_j)-(K\ast f)(x_i)\Big)H_{N,m-2}(z_{[m]\setminus\{i,j\}})\\
-\frac1N\sum_{i\ne j}^mf(z_j)\Big(K(x_i,x_j)-(K\ast f)(x_i)\Big)\cdot\nabla_{v_i}H_{N,m-1}(z_{[m]\setminus\{j\}})\\
-\frac1N\sum_{i\ne j}^m(\nabla_vf)(z_i)\cdot K(x_i,x_j)H_{N,m-1}(z_{[m]\setminus\{i\}})
+\frac mN\sum_{i=1}^m(\nabla_vf)(z_i)\cdot (K\ast f)(x_i) H_{N,m-1}(z_{[m]\setminus\{i\}})\\
+\frac{1}N\sum_{i\ne j}^mf(z_j)(\nabla f)(z_i)\cdot\int_\Xd K(x_i,x_*)H_{N,m-1}(z_{[m]\setminus\{i,j\}},z_*)\,dz_*\\
-\frac{N-m}N\sum_{i=1}^m\int_\Xd K(x_i,x_*)\cdot\nabla_{v_i}H_{N,m+1}(z_{[m]},z_*)\,dz_*,
\end{multline*}
where we set $H_{N,m}\equiv0$ for $m>N$ or $m<0$, and where we have defined for abbreviation
\begin{equation}\label{eq:def-LNm}
L_{N,m}:=\sum_{i=1}^mv_i\cdot\nabla_{x_i}-\sum_{i=1}^m\triangle_{v_i}
+\frac{N-m}N\sum_{i=1}^mK\ast f(x_i)\cdot\nabla_{v_i}
+\frac1N\sum_{i\ne j}^mK(x_i,x_j)\cdot\nabla_{v_i}.
\end{equation}
Moreover, chaotic initial data~\eqref{eq:chaotic-data} yield $H_{N,m}|_{t=0}=0$ for all $1\le m\le N$.
\end{lem}

\begin{proof}
Starting point is the BBGKY hierarchy~\eqref{eq:BBGKY0} for marginals $\{F_{N,m}\}_{1\le m\le N}$.
By definition of~$H_{N,m}$, the BBGKY equation for $F_{N,m}$ and the mean-field equation~\eqref{eq:VFP} for $f$ lead to
\begin{multline}\label{eq:pre-eqn-HNm}
\partial_tH_{N,m}+\sum_{i=1}^mv_i\cdot\nabla_{x_i}H_{N,m}
-\sum_{i=1}^m\triangle_{v_i}H_{N,m}\\
=-\frac1N\sum_{k=0}^m(-1)^{m-k}\sum_{\sigma\in P^m_k}f^{\otimes m-k}(z_{[m]\setminus\sigma})\sum_{i,j\in\sigma\atop i\ne j}K(x_i,x_j)\cdot\nabla_{v_i}F_{N,k}(z_\sigma)\\
-\sum_{k=0}^m(-1)^{m-k}\sum_{\sigma\in P^m_k}f^{\otimes m-k}(z_{[m]\setminus\sigma})\frac{N-k}N\sum_{i\in\sigma}\int_\Xd K(x_i,x_*)\cdot\nabla_{v_i}F_{N,k+1}(z_\sigma,z_*)\,dz_*\\
-\sum_{k=0}^m(-1)^{m-k}\sum_{\sigma\in P^m_k}\sum_{i\notin\sigma}F_{N,k}(z_\sigma)K\ast f(z_i)\cdot\nabla_{v_i}f^{\otimes m-k}(z_{[m]\setminus\sigma}).
\end{multline}
It remains to reformulate the right-hand side in terms of linear correlations. We start with the first right-hand side term.
Using the cluster expansion~\eqref{eq:cluster-FNmlin} to write marginals in terms of linear correlations, and reorganizing the sums, we get 
\begin{eqnarray*}
\lefteqn{-\frac1N\sum_{k=0}^m(-1)^{m-k}\sum_{\sigma\in P^m_k}f^{\otimes m-k}(z_{[m]\setminus\sigma})\sum_{i,j\in\sigma\atop i\ne j}K(x_i,x_j)\cdot\nabla_{v_i}F_{N,k}(z_\sigma)}\\
&=&-\frac1N\sum_{k=0}^m(-1)^{m-k}\sum_{\sigma\in P^m_k}\sum_{i,j\in\sigma\atop i\ne j}\sum_{l=0}^k\sum_{\tau\in P^\sigma_l}K(x_i,x_j)\cdot\nabla_{v_i}\Big(f^{\otimes m-l}(z_{[m]\setminus\tau})H_{N,l}(z_\tau)\Big)\\
&=&
-\frac1N\sum_{l=0}^m\sum_{\tau\in P^m_l}\sum_{i,j\in\tau\atop i\ne j}
f^{\otimes m-l}(z_{[m]\setminus\tau})K(x_i,x_j)\cdot\nabla_{v_i}H_{N,l}(z_\tau)
\sum_{k=l}^m(-1)^{m-k}\binom{m-l}{k-l}\\
&&-\frac1N\sum_{l=0}^m\sum_{\tau\in P^m_l}\sum_{i\in\tau}
\sum_{j\notin\tau}f^{\otimes m-l}(z_{[m]\setminus\tau})K(x_i,x_j)\cdot\nabla_{v_i}H_{N,l}(z_\tau)\sum_{k=l}^m(-1)^{m-k}\binom{m-l-1}{k-l-1}\\
&&-\frac1N\sum_{l=0}^m\sum_{\tau\in P^m_l}\sum_{i\notin\tau}\sum_{j\in\tau}H_{N,l}(z_\tau)K(x_i,x_j)\cdot\nabla_{v_i}f^{\otimes m-l}(z_{[m]\setminus\tau})\sum_{k=l}^m(-1)^{m-k}\binom{m-l-1}{k-l-1}\\
&&-\frac1N\sum_{l=0}^m\sum_{\tau\in P^m_l}\sum_{i,j\notin\tau\atop i\ne j}H_{N,l}(z_\tau)K(x_i,x_j)\cdot\nabla_{v_i}f^{\otimes m-l}(z_{[m]\setminus\tau})\sum_{k=l}^m(-1)^{m-k}\binom{m-l-2}{k-l-2},
\end{eqnarray*}
hence, using the combinatorial identity $\sum_{k=0}^n(-1)^{n-k}\binom{n}{k}=\mathds1_{n=0}$,
\begin{multline*}
-\frac1N\sum_{k=0}^m(-1)^{m-k}\sum_{\sigma\in P^m_k}f^{\otimes m-k}(z_{[m]\setminus\sigma})\sum_{i,j\in\sigma\atop i\ne j}K(x_i,x_j)\cdot\nabla_{v_i}F_{N,k}(z_\sigma)\\
\,=\,-\frac1N\sum_{i\ne j}^mK(x_i,x_j)\cdot\nabla_{v_i}H_{N,m}
-\frac1N\sum_{\tau\in P^m_{m-1}}\sum_{i\ne j}^mf(z_j)K(x_i,x_j)\cdot\nabla_{v_i}H_{N,m-1}(z_{[m]\setminus\{j\}})\\
-\frac1N\sum_{i\ne j}^mH_{N,m-1}(z_{[m]\setminus i})K(x_i,x_j)\cdot\nabla_{v}f(z_i)
-\frac1N\sum_{i\ne j}^mH_{N,l}(z_{[m]\setminus\{i,j\}})K(x_i,x_j)\cdot\nabla_{v}f(z_i)f(z_j).
\end{multline*}
Arguing similarly to rewrite the last two terms in~\eqref{eq:pre-eqn-HNm}, the conclusion follows. We skip the details for shortness.
\end{proof}

\subsection{Estimates on linear correlations}
A direct analysis of the above hierarchy leads to the following short-time optimal estimates on linear correlations. To deal with the case of an interaction kernel $K\in L^2$ deriving from a potential with exponential integrability, we take inspiration from the BBGKY hierarchical estimates in~\cite{BJS-22} using suitable time-dependent Gibbs-type weights. Note that particular care is needed here to ensure the existence of a fixed time interval on which linear correlation estimates hold to all orders: for this purpose, we shall appeal to Lemma~\ref{lem:hier}.

\begin{prop}\label{prop:HNm}$ $
In the setting of Theorem~\ref{th:main}, there is a time $T_*\in(0,T_0]$ and some $\beta_*>0$ such that we have for all $0\le m\le N$ and $t\in[0,T_*]$,
\[\Big(\int_{\Xd^m}|H_{N,m}(t)|^2\omega_{\beta_*}^{\otimes m}\Big)^\frac12\,\le\, (Cm^2)^mN^{-\frac m2}.\]
\end{prop}

\begin{proof}
Following~\cite{BJS-22}, given $\beta>0$ to be suitably chosen later on, we consider the following time-dependent Gibbs-type weights, for all $1\le m\le N$ and $t\ge0$,
\begin{equation*}
\omega_{N,m}(t;z_1,\ldots,z_m)
\,:=\,\exp\bigg[\frac{\beta}{1+4\beta t}\bigg(\frac12\sum_{i=1}^m|v_i|^2
+\frac{N-m}N\sum_{i=1}^m(W\ast f)(t,x_i)
+\frac1{2N}\sum_{i\ne j}^mW(x_i-x_j)\bigg)\bigg].
\end{equation*}
These weights are adapted to the linear operators $L_{N,m}$'s driving the correlation dynamics in the hierarchy derived in Lemma~\ref{lem:BBGKY-HNm}: indeed, by multiple integrations by parts, we obtain the following inequality for the resulting weighted dissipation, for all $1\le m\le N$, $h_m\in C^\infty_c(\R^+\times\Xd^m)$, and~$t\ge0$,
\begin{multline}\label{eq:weight-dissip}
2\int_{\Xd^m} h_m(L_{N,m}h_m)\omega_{N,m}-\int_{\Xd^m}|h_m|^2\partial_t\omega_{N,m}\\
\ge
2\sum_{i=1}^m\int_{\Xd^m} |\nabla_{v_i}h_m|^2\omega_{N,m}
+\frac{\beta^2}{(1+4\beta t)^2}\sum_{i=1}^m\int_{\Xd^m}|v_i|^2|h_m|^2\omega_{N,m}\\
-\frac{m\beta}{1+4\beta t}\Big(d+\|\partial_t(W\ast f)\|_{L^\infty}+\frac{6\beta \|W_-\|_{L^\infty}}{1+4\beta t}\Big)\int_{\Xd^m}|h_m|^2\omega_{N,m}.
\end{multline}
Consider the corresponding weighted $L^2$ norms of linear correlations, for $1\le m\le N$,
\[A_{N,m}\,:=\,\int_{\Xd^m}|H_{N,m}|^2\omega_{N,m}.\]
Performing an energy estimate for the hierarchy derived in Lemma~\ref{lem:BBGKY-HNm}, using the above dissipation inequality~\eqref{eq:weight-dissip}, and integrating by parts, we get for all~$t\in[0,T_0]$,
\begin{multline*}
\partial_tA_{N,m}\,\le\,
-2\sum_{i=1}^m\int_{\Xd^m} |\nabla_{v_i}H_{N,m}|^2\omega_{N,m}
-\frac{\beta^2}{(1+4\beta t)^2}\sum_{i=1}^m\int_{\Xd^m}|v_i|^2|H_{N,m}|^2\omega_{N,m}\\
+\frac{m\beta}{1+4\beta t}\Big(d+\|\partial_t(W\ast f)\|_{L^\infty}+\frac{6\beta \|W_-\|_{L^\infty}}{1+4\beta t}\Big)A_{N,m}\\
+\frac{2}N\sum_{i\ne j}^m\int_{\Xd^{m+1}}f(z_j) H_{N,[m+1]\setminus\{j\}}\omega_{N,[m]}K(x_i-x_{m+1})\cdot\Big(\nabla_{v_i}+\frac{\beta}{1+4\beta t}v_i\Big)H_{N,[m]}\\
+2\frac{N-m}N\sum_{i=1}^m\int_{\Xd^{m+1}}f(z_i) H_{N,[m+1]\setminus\{i\}}\omega_{N,[m]}K(x_i-x_{m+1})\cdot \Big(\nabla_{v_i}+\frac{\beta}{1+4\beta t}v_i\Big)H_{N,[m]}\\
+\frac1N\sum_{i\ne j}^m\int_{\Xd^m}f(z_i)f(z_j)H_{N,[m]\setminus\{i,j\}}\omega_{N,[m]}\Big(K(x_i-x_j)-(K\ast f)(x_i)\Big)\cdot\Big(\nabla_{v_i}+\frac{\beta}{1+4\beta t}v_i\Big)H_{N,[m]}\\
+\frac1N\sum_{i\ne j}^m\int_{\Xd^m}f(z_j)H_{N,[m]\setminus\{j\}}\omega_{N,[m]}\Big(K(x_i-x_j)-(K\ast f)(x_i)\Big)\cdot\Big(\nabla_{v_i}+\frac{\beta}{1+4\beta t}v_i\Big)H_{N,[m]}\\
+\frac1N\sum_{i\ne j}^m\int_{\Xd^m}f(z_i)H_{N,[m]\setminus\{i\}}\omega_{N,[m]}K(x_i-x_j)\cdot\Big(\nabla_{v_i}+\frac{\beta}{1+4\beta t}v_i\Big)H_{N,[m]}\\
-\frac mN\sum_{i=1}^m\int_{\Xd^m}f(z_i)H_{N,[m]\setminus\{i\}}\omega_{N,[m]}(K\ast f)(x_i)\cdot\Big(\nabla_{v_i}+\frac{\beta}{1+4\beta t}v_i\Big)H_{N,[m]}\\
-\frac{1}N\sum_{i\ne j}^m\int_{\Xd^{m+1}}f(z_i)f(z_j)\omega_{N,[m]} H_{N,[m+1]\setminus\{i,j\}}K(x_i-x_{m+1})\cdot\Big(\nabla_{v_i}+\frac{\beta}{1+4\beta t}v_i\Big)H_{N,[m]}\\
+\frac{N-m}N\sum_{i=1}^m\int_{\Xd^{m+1}} H_{N,[m+1]}\omega_{N,[m]} K(x_i-x_{m+1})\cdot\Big(\nabla_{v_i}+\frac{\beta}{1+4\beta t}v_i\Big)H_{N,[m]},
\end{multline*}
with the short-hand notation
\[H_{N,A}:=H_{N,\sharp A}(z_A),\qquad \omega_{N,A}:=\omega_{N,\sharp A}(z_A).\]
Using Young's inequality to absorb factors involving $\nabla_vH_{N,m}$ or $vH_{N,m}$ into the dissipation terms, we obtain
\begin{multline*}
\partial_tA_{N,m}\,\lesssim\,
m\Big(1+\|\partial_t(W\ast f)\|_{L^\infty}+\|W_-\|_{L^\infty}\Big)A_{N,m}\\
+\frac{m^3}{N^2}\int_{\Xd^{m}}\Big|\int_{\Xd}H_{N,[m+1]\setminus\{1\}}K(x_2-x_{m+1})\,dz_{m+1}\Big|^2f(z_1)^2\omega_{N,[m]}\,dz_{[m]}\\
+m\int_{\Xd^{m}}\Big|\int_{\Xd}H_{N,[m+1]\setminus\{1\}}K(x_1-x_{m+1})\,dz_{m+1}\Big|^2f(z_1)^2\omega_{N,[m]}\,dz_{[m]}\\
+\frac{m^3}{N^2}\int_{\Xd^m}|H_{N,[m]\setminus\{1,2\}}|^2\Big(|K(x_1-x_2)|^2+|(K\ast f)(x_1)|^2\Big)f(z_1)^2f(z_2)^2\omega_{N,[m]}\,dz_{[m]}\\
+\frac{m^3}{N^2}\int_{\Xd^m}|H_{N,[m]\setminus\{2\}}|^2\Big(|K(x_1-x_2)|^2+|(K\ast f)(x_1)|^2\Big)f(z_2)^2\omega_{N,[m]}\,dz_{[m]}\\
+\frac{m^3}{N^2}\int_{\Xd^m}|H_{N,[m]\setminus\{1\}}|^2\Big(|K(x_1-x_2)|^2+|(K\ast f)(x_1)|^2\Big)f(z_1)^2\omega_{N,[m]}\,dz_{[m]}\\
+\frac{m^3}{N^2}\int_{\Xd^{m}}\Big|\int_{\Xd}H_{N,[m+1]\setminus\{1,2\}}K(x_1-x_{m+1})\,dz_{m+1}\Big|^2f(z_1)^2f(z_2)^2\omega_{N,[m]}\,dz_{[m]}\\
+m\int_{\Xd^{m}}\Big|\int_\Xd H_{N,[m+1]}K(x_i-x_{m+1})\,dz_{m+1}\Big|^2\omega_{N,[m]}dz_{[m]}.
\end{multline*}
In order to bound the right-hand side in terms of the $A_{N,m}$'s themselves, we need to add or remove variables in the weights. More precisely, using estimates such as
\begin{multline*}
\omega_{N,[m]}
\,\le\, \omega_{N,[m+1]\setminus\{1\}}\,e^{3\beta\|W_-\|_{L^\infty}}e^{\beta\|W\ast f\|_{L^\infty}}\\[-2mm]
\times\exp\Big[-\frac{\beta}{2(1+4\beta t)}|v_{m+1}|^2\Big]
\exp\Big[\frac\beta2|v_1|^2\Big]\exp\Big[\frac\beta{N}\sum_{i=2}^mW(x_1-x_i)\Big],
\end{multline*}
we are easily led to the following: for all $0\le m\le N$ and $t\in[0,T_0]$,
\begin{equation}\label{eq:ANm-diff}
\left\{\begin{array}{l}
\partial_tA_{N,m}\,\le\,
CmA_{N,m}
+CmA_{N,m+1}
+C\frac{m^3}{N^2}(A_{N,m-1}+A_{N,m-2}),\\
A_{N,m}|_{t=0}=\mathds1_{m=0},
\end{array}\right.
\end{equation}
where the constant $C$ only depends on $T_0,\beta,$ and on
\begin{gather*}
\|W_-\|_{L^\infty},\qquad\|K\|_{L^2},\qquad\sup_{y\in\T^d}\int_{\T^d} e^{2\beta W(x-y)}|K(x)|^2\,dx,\\[-2mm]
\|W\ast f\|_{L^\infty(0,T_0;W^{1,\infty})},\qquad \|\partial_t(W\ast f)\|_{L^{\infty}(0,T_0;L^{\infty})},\qquad\sup_{t\in[0,T_0],x\in\T^d}\int_{\R^d}|f(t;x,v)|^2e^{\frac\beta2|v|^2}dv.
\end{gather*}
For a suitable choice of $\beta=\beta_*>0$, these quantities are all ensured to be finite by assumption.
We turn to the analysis of the hierarchy~\eqref{eq:ANm-diff} of differential inequalities. First, bounding $\frac{m^3}{N^2}\le m$, we get for all $0\le m\le N$ and $t\in[0,T_0]$,
\begin{equation*}
\left\{\begin{array}{l}
\partial_tA_{N,m}\,\le\,
Cm(A_{N,m-2}+A_{N,m-1}+A_{N,m}+A_{N,m+1}),\\
A_{N,m}|_{t=0}=\mathds1_{m=0},
\end{array}\right.
\end{equation*}
In terms of the generating function
\[Z(t,r)=\sum_{m=0}^Nr^mA_{N,m}(t),\qquad t\in[0,T_0],\quad 0\le r<1,\]
we get
\begin{equation*}
\partial_tZ(t,r)
\,\le\, C\partial_rZ(t,r),\qquad Z(0,r)=1.
\end{equation*}
Solving this differential inequality yields for all $0\le r<1$ and $0\le t\le (C^{-1}r)\wedge T_0$,
\begin{equation*}
Z(t,r-Ct)
\le Z(0,r)\,=\,1,
\end{equation*}
which implies for all $0\le m\le N$ and $t\in[0,T_*]$, with $T_*:=(2C)^{-1}\wedge T_0$,
\[A_{N,m}(t)\,\le\, 2^m.\]
Next, taking advantage of this exponential a priori control, we can now go back to the differential inequality~\eqref{eq:ANm-diff} and apply the hierarchical estimate of Lemma~\ref{lem:hier}: we conclude for all $1\le m\le N$ and~$t\in[0,T_*]$,
\[A_{N,m}(t)\,\le\, (Cm^2)^mN^{-m}.\qedhere\]
\end{proof}

\subsection{Nonlinear correlations}
We turn to the study of the usual (nonlinear) correlation functions $\{G_{N,m}\}_{1\le m\le N}$ defined in~\eqref{eq:correl}.
We start by unraveling their link to linear correlations. In a nutshell, the following lemma shows that linear correlations satisfy the same cluster expansion in terms of nonlinear correlations as marginals do, cf.~\eqref{eq:cluster-exp}, up to replacing $G_{N,1}=F_{N,1}$ by $F_{N,1}-f$.

\begin{lem}\label{lem:rel-G-H}
For all $1\le m\le N$,
\[H_{N,m}=\sum_{\pi\vdash[m]}\prod_{B\in\pi}\tilde G_{N,\sharp B}(z_B).\]
where we have set $\tilde G_{N,m}:=G_{N,m}$ for $m>1$ and $\tilde G_{N,1}:=F_{N,1}-f$.
\end{lem}

\begin{proof}
Recalling the definition of $H_{N,m}$ and inserting cluster expansions for marginals in terms of nonlinear correlations, cf.~\eqref{eq:cluster-exp}, we find
\begin{eqnarray*}
H_{N,m}&=&\sum_{k=0}^m(-1)^{m-k}\sum_{\sigma\in P^m_k}F_{N,k}(z_\sigma)f^{\otimes m-k}(z_{[m]\setminus\sigma})\\
&=&\sum_{k=0}^m(-1)^{m-k}\sum_{\sigma\in P^m_k}\sum_{\pi\vdash\sigma}\prod_{B\in\pi}G_{N,\sharp B}(z_B)f^{\otimes m-k}(z_{[m]\setminus\sigma}),
\end{eqnarray*}
and the conclusion follows after a straightforward recombination.
\end{proof}

From the above lemma, we find that Proposition~\ref{prop:HNm} implies the following estimates on nonlinear correlations, thus proving the first part~\eqref{eq:L2bnd} of Theorem~\ref{th:main}.

\begin{cor}\label{cor:estim-GNm-0}$ $
In the setting of Theorem~\ref{th:main}, there is a time $T_*\in(0,T_0]$ and some $\beta_*>0$ such that we have for all $2\le m\le N$ and $t\in[0,T_*]$,
\begin{eqnarray*}
\Big(\int_{\Xd}|(F_{N,1}-f)(t)|^2\omega_{\beta_*}\Big)^\frac12&\le&CN^{-\frac12},\\
\Big(\int_{\Xd^m}|G_{N,m}(t)|^2\omega_{\beta_*}^{\otimes m}\Big)^\frac12&\le&(Cm^2)^mN^{-\frac m2}.
\end{eqnarray*}
\end{cor}

\begin{proof}
Using the same notation as in Lemma~\ref{lem:rel-G-H}, we need to show for all $1\le m\le N$ and $t\in[0,T_*]$,
\[\Big(\int_{\Xd^m}|\tilde G_{N,m}(t)|^2\omega_{\beta_*}^{\otimes m}\Big)^\frac12\,\le\,(Cm^2)^mN^{-\frac m2}.\]
We argue by induction. For $m=1$, we have $\tilde G_{N,1}=F_{N,1}-f=H_{N,1}$, so the conclusion already follows from Proposition~\ref{prop:HNm}. Next, we assume that for some $n\ge1$ the result is already known to hold for all $1\le m\le n$.
Appealing to the previous lemma in form of
\[\tilde G_{N,n+1}=H_{N,n+1}-\sum_{\pi\vdash[n+1]\atop\sharp\pi>1}\prod_{B\in\pi}\tilde G_{N,\sharp B}(z_B),\]
we can estimate
\begin{equation*}
\Big(\int_{\Xd^{n+1}} |\tilde G_{N,n+1}|^2\omega_{\beta_*}^{\otimes n+1}\Big)^\frac12
\,\le\,\Big(\int_{\Xd^{n+1}} |H_{N,n+1}|^2 \omega_{\beta_*}^{\otimes n+1}\Big)^\frac12
+\sum_{\pi\vdash[n+1]\atop\sharp\pi>1}\prod_{B\in\pi}\Big(\int_{\Xd^{\sharp B}} |\tilde G_{N,\sharp B}|^2\omega_{\beta_*}^{\otimes \sharp B}\Big)^\frac12.
\end{equation*}
and the conclusion follows from Proposition~\ref{prop:HNm} for $H_{N,n+1}$ and the induction assumption.
\end{proof}

\subsection{Weak limit of subcritically rescaled correlations}
In this section, we turn to the proof of the second part~\eqref{eq:wL2bnd} of Theorem~\ref{th:main}, that is, for all $m\ge3$,
\[N^{\frac m2}G_{N,m}\,\overset*\rightharpoonup\,0,\qquad\text{in $L^\infty(0,T_*;L^2(\omega_{\beta_*}^{\otimes m}))$}.\]
For that purpose, we examine the hierarchy of equations satisfied by nonlinear correlations, and we pass to the limit in this hierarchy. We emphasize that this is only possible because we already have the subcritical bounds~\eqref{eq:L2bnd} on the $G_{N,m}$'s. In that regard, the detour through the $H_{N,m}$'s and their hierarchy is key to our approach.

Before we formulate the hierarchy of equations satisfied by nonlinear correlations, we start with some useful notation.

\begin{defin}
Consider a collection $\{h_m\}_{1\le m\le N}$ of functions $h_m:\Xd^m\to\R$ such that for all $m$ the function $h_m$ is symmetric in its $m$ entries (such as $\{G_{N,m}\}_{1\le m\le N}$).
\begin{enumerate}[---]
\item For $P\subset[ m]$ with $P\ne\varnothing$, we define $h_P:=h_{\sharp P}(z_P)$, and for $P=\varnothing$ we set $h_\varnothing:=0$ (as $h_m$ is undefined for $m<1$).
For $P\subset[ m]$ with $P\ne[ N]$, we define $h_{P\cup\{*\}}:=h_{\sharp P+1}(z_P,z_*)$, and for $P=[ N]$ we set $h_{[ N]\cup\{*\}}=0$ (as $h_m$ is undefined for $m>N$).
\smallskip\item For $P\subset[ m]$ and $k,\ell\in P$, we define
\[S_{k,\ell}h_P\,:=\,-K(x_k - x_\ell) \cdot \nabla_{v_k} h_P.\]
\item For $P\subset[ m]$ and $k\in P$, we define
\[H_kh_{P\cup\{*\}}\,:=\,-\int_{\Xd}K(x_k-x_*)\cdot\nabla_{v_k}h_{P\cup\{*\}}\,dz_*.\]
\end{enumerate}
\end{defin}

We may now formulate the BBGKY hierarchy for nonlinear correlations. The derivation follows similarly as in the proof of Lemma~\ref{lem:BBGKY-HNm}, and we refer to~\cite[Section~4]{Hess_Childs_2023} for the corresponding derivation in case of the overdamped dynamics; we skip the details for shortness. For convenience, we write $A-B=A\setminus B$ for set difference, with for instance $A-B-C=A\setminus(B\cup C)$ and $A\cup B-C=(A\cup B)\setminus C$.
Due to terms involving $S_{k,\ell}$, we emphasize that this hierarchy cannot be used to obtain estimates unless~$K\in L^\infty$; see~Section~\ref{sec:Linfty}.

\begin{lem}[Hierarchy for nonlinear correlations]\label{lem:hier-corr}
For fixed $N$, correlation functions $\{G_{N,m}\}_{1 \le m\le N}$ satisfy the following hierarchy of equations: for all $1\le m\le N$,
\begin{align*}
\partial_t G_{N,m}+&\, L_{N,m} G_{N,m}  = -\frac{N-m}N\sum_{i=1}^m(\nabla_{v}f)(z_i)\cdot\int_{\Xd}K(x_i-x_*)G_{N,m}(z_{[m]\setminus\{i\}},z_*)\,d z_*\\
&+\frac{N-m}{N} \sum_{k=1}^m H_k G_{N,[ m ] \cup \{\ast\}}
-  \sum_{k=1}^m \sum_{A \subset [ m ] - \{k\}} \frac{m-1-\sharp A}{N} H_k\big(G_{N,A \cup \{k, \ast\}} G_{N,[ m ] - \{k\} - A} \big)\\
&+  \frac{N-m}{N} \sum_{k=1}^m \sum_{A \subsetneq [ m ] - \{k\}\atop A\ne\varnothing} H_k\big( G_{N,A \cup \{k\}} G_{N,[ m ] \cup \{\ast\} - A - \{k\}} \big) \\
&-  \sum_{k=1}^m \sum_{A \subset [ m] - \{k\}} \sum_{B\subset [ m] - \{k\} - A} \frac{m-1-\sharp A-\sharp B}{N} H_k\big( G_{N,A \cup \{k\}} G_{N,B \cup \{*\}} G_{N,[ m] - A-B - \{k\}} \big) \\
&+ \frac{1}{N} \sum_{k \ne \ell}^m \sum_{A \subset [ m] - \{k , \ell \}} S_{k, \ell}\big( G_{N,A \cup \{k\}} G_{N,[ m] - A - \{k\}}\big) \\
&+ \frac{N-m}{N} \sum_{k=1}^m H_k \Big( (F^{N,\{*\}}-f(z_*))\, G_{N,[ m ]} + (F^{N,\{k\}}-f(z_k))\, G_{N, [ m ] \cup \{*\}-\{k\}} \Big),
\end{align*}
where we recall that $L_{N,m}$ is defined in~\eqref{eq:def-LNm}.
\end{lem}

With this hierarchy at hand, we can now conclude the proof of~\eqref{eq:wL2bnd}.
By Corollary~\ref{cor:estim-GNm-0}, we obtain the following by weak compactness: up to a subsequence as $N\uparrow\infty$, for all $m\ge2$,
\begin{equation}\label{eq:conv-GNm}
N^{\frac m2}G_{N,m}\overset*\rightharpoonup \bar G_m,\qquad\text{in $L^\infty(0,T_*;L^2(\omega_{\beta_*}^{\otimes m}))$},
\end{equation}
for some limit $\{\bar G_m\}_{m\ge2}$. Also note that Corollary~\ref{cor:estim-GNm-0} yields, for $m=1$,
\[F_{N,1}\to f,\qquad\text{in $L^\infty(0,T_*;L^2(\omega_{\beta_*}))$}.\]
Now passing to the weak limit in the hierarchy of Lemma~\ref{lem:hier-corr}, which is allowed for $K\in L^2$, we find
\begin{multline}\label{eq:conv-lim-GNm}
\partial_t \bar G_{m}
+\sum_{i=1}^mv_i\cdot\nabla_{x_i}\bar G_{m}
-\sum_{i=1}^m\triangle_{v_i}\bar G_{m}
+\sum_{i=1}^m (K\ast f)(z_i)\cdot\nabla_{v_i}\bar G_{m}\\
+\sum_{i=1}^m(\nabla_{v}f)(z_i)\cdot\int_{\Xd}K(x_i-x_*)\bar G_{m}(z_{[m]\setminus\{i\}},z_*)\,d z_*\\
\,=\,
-\mathds1_{m=2}\sum_{k\ne\ell}^2\Big(K(x_k-x_\ell)-(K\ast f)(x_k)\Big)\cdot(\nabla_vf)(z_k)f(z_\ell),
\end{multline}
with $\bar G^m|_{t=0}=0$ for $m\ge2$.
This implies $\bar G^m\equiv0$ for all $m\ge3$ and concludes the proof of~\eqref{eq:wL2bnd}.

\begin{rem}
For comparison, considering weak limits of rescaled linear correlations,
\[N^\frac m2H_{N,m}\overset*\rightharpoonup \bar H_m,\qquad\text{in $L^\infty(0,T_*;L^2(\omega_{\beta_*}^{\otimes m}))$},\]
and passing to the limit in the corresponding linear hierarchy in Lemma~\ref{lem:BBGKY-HNm}, we find for all $m\ge1$,
\begin{multline*}
\partial_t\bar H_{m}
+\sum_{i=1}^mv_i\cdot\nabla_{x_i}\bar H_{m}-\sum_{i=1}^m\triangle_{v_i}\bar H_{m}+\sum_{i=1}^m(K\ast f)(x_i)\cdot\nabla_{v_i}\bar H_{m}\\
+\sum_{i=1}^m(\nabla_{v} f)(z_i)\cdot\int_\Xd K(x_i-x_*)\bar H_{m}(z_{[m]\setminus\{i\}},z_*)\,dz_*\\
=-\sum_{i\ne j}^mf(z_j)(\nabla_vf)(z_i)\cdot\Big(K(x_i-x_j)-(K\ast f)(x_i)\Big)\bar H_{m-2}(z_{[m]\setminus\{i,j\}}),
\end{multline*}
with initially $\bar H_m|_{t=0}=0$ for all $m\ge1$. This immediately implies $\bar H_m\equiv0$ for $m$ odd, while terms of even order satisfy a nontrivial hierarchy. In fact, the latter implies relations of the form $\bar H_4=3\,\Sym(\bar H_2\otimes\bar H_2)$, and so on, which are equivalent to the vanishing of limit nonlinear correlations~$\{\bar G_m\}_m$.
\end{rem}

\section{Bogolyubov correction}
This section is devoted to the proof of Corollary~\ref{cor:Bogo}, which we split into four steps.

\medskip\noindent
{\bf Step~1:} Proof that for all $1\le m\le N$ and $t\in[0,T_*]$,
\begin{equation}\label{eq:prop-chaos}
\|F_{N,m}-f^{\otimes m}\|_{L^2(\omega_{\beta_*}^{\otimes m})}\,\le\,C_mN^{-1}.
\end{equation}
Note that for $m=1$ a suboptimal rate $O(N^{-1/2})$ was obtained in Corollary~\ref{cor:estim-GNm-0}, which we now improve.
The starting point is the BBGKY hierarchy~\eqref{eq:BBGKY0} for marginals. Comparing it to the mean-field equation for $f$, we find
\begin{multline}\label{eq:FNm-fm}
\partial_t(F_{N,m}-f^{\otimes m})+\sum_{i=1}^mv_i\cdot\nabla_{x_i}(F_{N,m}-f^{\otimes m})
-\sum_{i=1}^m\triangle_{v_i}(F_{N,m}-f^{\otimes m})\\
+\frac1N\sum_{i\ne j}^mK(x_i-x_j)\cdot\nabla_{v_i}(F_{N,m}-f^{\otimes m})\\
=-\frac{N-m}N\sum_{i=1}^m\int_\Xd K(x_i-x_*)\cdot\nabla_{v_i}(F_{N,m+1}-f^{\otimes m+1})(\cdot,z_*)\,dz_*\\
-\frac1N\sum_{i, j=1}^m\big(K(x_i-x_j)\mathds1_{i\ne j}-K\ast f(x_i)\big)\cdot\nabla_{v_i}f^{\otimes m}.
\end{multline}
Similarly as in the proof of Proposition~\ref{prop:HNm}, we perform $L^2$ estimates with suitable time-dependent weights. More precisely, we now define for all $1\le m\le N$ and $t\ge0$,
\[\omega_{N,m}'(t;z_{[m]})\,=\,\exp\bigg[\frac\beta{1+4\beta t}\bigg(\frac12\sum_{i=1}^m|v_i|^2+\frac1{2N}\sum_{i\ne j}^mW(x_i-x_j)\bigg)\bigg].\]
Performing a weighted energy estimate for the above hierarchy~\eqref{eq:FNm-fm}, with these weights, we find
\begin{multline*}
\partial_t\int_{\Xd^m}|F_{N,m}-f^{\otimes m}|^2\omega_{N,m}'\\
\,\le\,
-2\int_{\Xd^m}\sum_{i=1}^m|\nabla_{v_i}(F_{N,m}-f^{\otimes m})|^2\omega_{N,m}'
-\frac{\beta^2}{(1+4\beta t)^2}\sum_{i=1}^m\int_{\Xd^m}|v_i|^2|F_{N,m}-f^{\otimes m}|^2\omega_{N,m}'\\
+\frac{m\beta}{1+4\beta t}\Big(d+\frac{2\beta\|W_-\|_{L^\infty}}{1+4\beta t)}\Big)\int_{\Xd^m}|F_{N,m}-f^{\otimes m}|^2\omega_{N,m}'\\
+2\frac{N-m}N\sum_{i=1}^m\int_{\Xd^{m+1}}\omega_{N,m}'(z_{[m]})(F_{N,m+1}-f^{\otimes m+1})(z_{[m+1]})K(x_i-x_{m+1})\\
\hspace{5cm}\cdot\Big(\nabla_{v_i}+\frac{\beta}{1+4\beta t}v_i\Big)(F_{N,m}-f^{\otimes m})(z_{[m]})\,dz_{[m+1]} \\
-\frac2N\sum_{i, j=1}^m\int_{\Xd^m}\omega_{N,m}'f^{\otimes m}\big(K(x_i-x_j)\mathds1_{i\ne j}-K\ast f(x_i)\big)\cdot\Big(\nabla_{v_i}+\frac\beta{1+4\beta t}v_i\Big)(F_{N,m}-f^{\otimes m}).
\end{multline*}
By Young's inequality, using the dissipation, we deduce
\begin{multline*}
\partial_t\int_{\Xd^m}|F_{N,m}-f^{\otimes m}|^2\omega_{N,m}'
\,\lesssim\,
m\int_{\Xd^m}|F_{N,m}-f^{\otimes m}|^2\omega_{N,m}'\\
+m\int_{\Xd^{m}}\Big|\int_\Xd (F_{N,m+1}-f^{\otimes m+1})(z_{[m+1]})\,K(x_1-x_{m+1})\,dz_{m+1}\Big|^2\omega_{N,m}'(z_{[m]})\,dz_{[m]}\\
+\frac{m^3}{N^2}\int_{\Xd^m}\Big(|K(x_1-x_2)|^2+|K\ast f(x_1)|^2\Big)|f^{\otimes m}|^2\omega_{N,m}'.
\end{multline*}
Changing the weight in the second term,
\[\omega_{N,m}'\,\lesssim\,\omega_{N,m+1}'\exp\Big[-\frac\beta{2(1+4\beta t)}|v_{m+1}|^2\Big],\]
and recalling the assumptions on $K,f$, for a suitable choice of $\beta=\beta_*>0$, we are led to
\begin{multline*}
\partial_t\int_{\Xd^m}|F_{N,m}-f^{\otimes m}|^2\omega_{N,m}'\\[-2mm]
\,\lesssim\,
m\int_{\Xd^m}|F_{N,m}-f^{\otimes m}|^2\omega_{N,m}'
+m\int_{\Xd^{m+1}}|F_{N,m+1}-f^{\otimes m+1}|^2\omega_{N,m+1}'
+\frac{m^3}{N^2}.
\end{multline*}
In terms of the generating function
\[Z(t,r)\,=\,\sum_{m=1}^Nr^m\bigg(\int_{\Xd^m}|F_{N,m}-f^{\otimes m}|^2\omega_{N,m}'+\frac{m^2}{N^2}\bigg),\qquad t\in[0,T_0],\quad 0\le r<1,\]
the above yields
\begin{equation*}
\partial_tZ(t,r)\,\le\,
C\partial_rZ(t,r),\qquad Z(0,r)\lesssim(1-r)^{-3}N^{-2}.
\end{equation*}
Solving this differential inequality, the claim~\eqref{eq:prop-chaos} follows.

\medskip\noindent
{\bf Step~2:} Proof that
\begin{equation}\label{eq:conv-GNm-re}
NG_{N,2}\overset*\rightharpoonup \bar G_2,\qquad\text{in $L^\infty(0,T_*;L^2(\omega_{\beta_*}^{\otimes 2}))$},
\end{equation}
where $\bar G_2$ satisfies
\begin{multline}\label{eq:conv-lim-GNm-re}
\partial_t \bar G_2
+\sum_{i=1}^2v_i\cdot\nabla_{x_i}\bar G_2
-\sum_{i=1}^2\triangle_{v_i}\bar G_2
+\sum_{i=1}^2 (K\ast f)(z_i)\cdot\nabla_{v_i}\bar G_{2}\\
+\sum_{i=1}^2(\nabla_{v}f)(z_i)\cdot\int_{\Xd}K(x_i-x_*)\bar G_2(z_{[2]\setminus\{i\}},z_*)\,d z_*\\
\,=\,
-\sum_{k\ne\ell}^2\Big(K(x_k-x_\ell)-(K\ast f)(x_k)\Big)\cdot(\nabla_vf)(z_k)f(z_\ell),
\end{multline}
with $\bar G_2|_{t=0}=0$.
In the previous section, cf.~\eqref{eq:conv-GNm}--\eqref{eq:conv-lim-GNm} for $m=2$, we already showed that the convergence~\eqref{eq:conv-GNm-re} holds up to a subsequence as $N\uparrow\infty$ and that any weak limit satisfies the above equation~\eqref{eq:conv-lim-GNm-re}. As the latter is easily checked to be well-posed in $L^2(\omega_{\beta_*}^{\otimes2})$ under the assumptions on~$K,f$, the conclusion follows.

\medskip\noindent
{\bf Step~3:} Proof that
\[N(F_{N,1}-f)\overset*\rightharpoonup h\qquad\text{in $L^\infty(0,T_*;L^2(\omega_{\beta_*}))$},\]
where $h$ satisfies
\begin{multline}\label{eq:defh}
\partial_th+v\cdot\nabla_{x}h
=\triangle_{v}h
-\int_\Xd K(x-x_*)\cdot\nabla_{v}\bar G_{2}(z,z_*)\,dz_*\\
- (K\ast h)\cdot\nabla_{v}f
- (K\ast f)\cdot\nabla_{v}h
+(K\ast f)\cdot\nabla_{v}f.
\end{multline}
By weak compactness, the bounds of Step~1 ensure the following convergence, up to a subsequence as~$N\uparrow\infty$,
\[N(F_{N,1}-f)\overset*\rightharpoonup h,\qquad\text{in $L^\infty(0,T_*;L^2(\omega_{\beta_*}))$},\]
for some limit $h$.
For $m=1$, the BBGKY equation~\eqref{eq:FNm-fm} for $F_{N,1}-f$ takes the form
\begin{multline*}
\partial_t(F_{N,1}-f)+v\cdot\nabla_{x}(F_{N,1}-f)
-\triangle_{v}(F_{N,1}-f)
=-\frac{N-1}N\int_\Xd K(x-x_*)\cdot\nabla_{v}G_{N,2}(z,z_*)\,dz_*\\
-\frac{N-1}N (K\ast(F_{N,1}-f))\cdot\nabla_{v}F_{N,1}
-\frac{N-1}N (K\ast f)\cdot\nabla_{v}(F_{N,1}-f)
+\frac1N (K\ast f)\cdot\nabla_{v}f.
\end{multline*}
Multiplying by $N$, passing to the weak limit, and recalling the strong convergence $F_{N,1}\to f$ in $L^\infty(0,T_*;L^2(\omega_{\beta_*}))$, cf.~Corollary~\ref{cor:estim-GNm-0}, as well as the weak convergence~\eqref{eq:conv-GNm-re} for correlations, we then find that any limit point~$h$ satisfies the desired equation~\eqref{eq:defh}. Well-posedness of the latter in $L^2(\omega_{\beta_*})$ yields the conclusion.

\medskip\noindent
{\bf Step~4:} Proof that for all $m\ge1$,
\[\qquad N\Big(F_{N,m}-(f+\tfrac1Nh)^{\otimes m}\Big)
\overset*\rightharpoonup\sum_{1\le k<\ell\le m}\bar G_2(z_k,z_\ell)f^{\otimes m-2}(z_{[m]\setminus\{k,\ell\}}),\qquad\text{in $L^\infty(0,T_*;L^2(\omega_{\beta_*}^{\otimes m}))$.}\]
Using the cluster expansion~\eqref{eq:cluster-exp} in form of
\[F_{N,m}-F_{N,1}^{\otimes m}=\sum_{\pi\vdash[m]\atop\sharp\pi<m}\prod_{B\in\pi}G_{N,B},\]
using bounds on correlations, cf.~Corollary~\ref{cor:estim-GNm-0}, as well as the weak convergence~\eqref{eq:conv-GNm-re}, we deduce
\[N(F_{N,m}-F_{N,1}^{\otimes m})\overset*\rightharpoonup\sum_{1\le k<\ell\le m}\bar G_{2}(z_k,z_\ell)f^{\otimes m-2}(z_{[m]\setminus\{k,\ell\}}).\]
Combining this with the result of Step~3, we deduce
\begin{multline*}
N(F_{N,m}-f^{\otimes m})
=N(F_{N,m}-F_{N,1}^{\otimes m})+N(F_{N,1}^{\otimes m}-f^{\otimes m})\\
\overset*\rightharpoonup
\sum_{1\le k<\ell\le m}\bar G_{2}(z_k,z_\ell)f^{\otimes m-2}(z_{[m]\setminus\{k,\ell\}})
+\sum_{k=1}^mf^{\otimes k-1}\otimes h\otimes f^{\otimes m-k},
\end{multline*}
and the claim follows.\qed

\section{Fluctuations of empirical measure}
This section is devoted to the proof of Corollary~\ref{cor:CLT}.
Let $\varphi\in C^\infty_c(\Xd)$ be a fixed test function. We split the proof into two steps.

\medskip\noindent
{\bf Step~1:} Convergence of the variance.\\
The variance of the empirical measure can be written as follows in terms of the $2$-point correlation function,
\[\Var\bigg[\int_\Xd\varphi\mu_N\bigg]=\frac{N-1}N\int_{\Xd^2}\varphi^{\otimes2}G_{N,2}+\frac1N\int_\Xd\Big(\varphi-\int_\Xd\varphi F_{N,1}\Big)^2F_{N,1}.\]
The convergence results of Corollary~\ref{cor:Bogo} then lead to the following characterization of the limiting variance,
\[\sigma(\varphi)^2\,:=\,\lim_{N\uparrow\infty}\Var\bigg[N^\frac12\int_\Xd\varphi\mu_N\bigg]\,=\,\int_{\Xd^2}\varphi^{\otimes2}\bar G_2+\int_\Xd\Big(\varphi-\int_\Xd\varphi f\Big)^2f,\]
where we recall that $\bar G_2$ is the solution of the Bogolyubov equation~\eqref{eq:Bogo}.
Similarly as e.g.\@ in~\cite[Section~5]{BD_2024}, we can easily check that this limiting variance can be reformulated as $\sigma(\varphi)^2=\Var[\int_\Xd\varphi\nu]$, where $\nu$ is the Gaussian process defined in the statement.

\medskip\noindent
{\bf Step~2:} CLT by the moment method.\\
From e.g.~\cite[Proposition~4.1]{D-21} (see also~\cite[Lemma~2.6]{BD_2024}), we recall the following link between correlation functions and cumulants of the empirical measure: for all $1\le m\le N$,
\begin{equation*}
\bigg|\kappa^m\bigg[\int_{\Xd}\varphi \mu_N\bigg]\bigg|
\,\lesssim_m\,\sum_{\pi\vdash\llbracket m\rrbracket}\sum_{\rho\vdash\pi}N^{\sharp\pi-\sharp\rho-m+1}\bigg|\int_{\Xd^{\sharp\pi}}\Big(\bigotimes_{B\in\pi}\varphi^{\sharp B}\Big)\Big(\prod_{D\in\rho}G_{N,\sharp D}(z_D)\Big)dz_\pi\bigg|.
\end{equation*}
Rescaling in $N$, counting the powers of $N$, and appealing to the weak convergence result~\eqref{eq:wL2bnd} of Theorem~\ref{th:main}, we deduce for all $m\ge3$, on $[0,T_*]$,
\[\lim_{N\uparrow\infty}\kappa^m\bigg[N^\frac12\int_{\Xd}\varphi \mu_N\bigg]\,=\,0.\]
By definition of cumulants, combining this with the convergence of the variance in Step~1, we deduce for all $m\ge1$, on $[0,T_*]$,
\[\lim_{N\uparrow\infty}\E\bigg[\Big(N^\frac12\int_{\Xd}\varphi (\mu_N-\E[\mu_N])\Big)^m\bigg]\,=\,\sigma(\varphi)^m\theta_m,\]
where $\theta_m$ stands for the $m$-th moment of a standard Gaussian variable. As $\nu$ is a centered Gaussian process with $\sigma(\varphi)^2=\Var[\int_\Xd\varphi\nu]$, note that $\sigma(\varphi)^m\theta_m=\E[(\int_\Xd\varphi\nu)^m]$.
By the moment method, a sequence of random variables converges in law to a Gaussian if and only if all its moments converge. This concludes the proof.\qed

\section{Case of bounded forces}\label{sec:Linfty}
In this section, we turn to the proof of Theorem~\ref{th:main-Linfty}, explaining how our correlation estimates can be improved a posteriori in the case $K\in L^\infty$.
We start by adapting Proposition~\ref{prop:HNm}, showing that linear correlation estimates now hold globally in time up to some exponential time growth.

\begin{prop}\label{prop:HNm-re}
In the setting of Theorem~\ref{th:main-Linfty}, we have for all $1\le m\le N$ and $t\ge0$,
\[\Big(\int_{\Xd^m}|H_{N,m}(t)|^2\omega_{\beta(t)}^{\otimes m}\Big)^\frac12\,\le\,(Ce^{Ct}m^2)^mN^{-\frac m2},\]
where we have set $\beta(t):=\frac\beta{1+4\beta t}$, for some constant $C$ only depending on $d,\beta,\|K\|_{L^\infty},\|f_\circ\|_{L^2(\omega_\beta)}$.
\end{prop}

\begin{proof}
Compared to the proof of Proposition~\ref{prop:HNm}, as we now assume $K\in L^\infty(\T^{2d})^d$, we can replace the time-dependent weight $\omega_{N,m}$ by the following simpler version,
\begin{equation}\label{eq:def-tildeomem}
\omega_{m}(t;z_{[m]})
\,:=\,\exp\bigg[\frac{\beta}{2(1+4\beta t)}\sum_{i=1}^m|v_i|^2\bigg]\,=\,\omega_{\beta(t)}^{\otimes m}(z_{[m]}),
\end{equation}
and we start by noting that we have global exponential a priori bounds on weighted $L^2$ norms of marginals. This follows from~\cite{BJS-22}, but we include a short proof for convenience: by a direct computation starting from the BBGKY equations~\eqref{eq:BBGKY0} for marginals, we can estimate
\begin{multline*}
\partial_t\int_{\Xd^m} |F_{N,m}|^2\omega_{m}
\,\le\,-2\sum_{i=1}^m\int_{\Xd^m} |\nabla_{v_i}F_{N,m}|^2\omega_{m}
-\frac{\beta^2}{(1+4\beta t)^2}\sum_{i=1}^m\int_{\Xd^m} |v_i|^2|F_{N,m}|^2\omega_{m}\\
+\frac{\beta dm}{1+4\beta t}\int_{\Xd^m} |F_{N,m}|^2 \omega_{m}
+\frac{\beta\|K\|_{L^\infty}}{1+4\beta t}\frac{m}N\sum_{i=1}^m\int_{\Xd^m} |v_i| |F_{N,m}|^2 \omega_{m}\\
+2\frac{N-m}N\sum_{i=1}^m\int_{\Xd^{m}}\omega_{m} \Big(\int_{\Xd} F_{N,m+1}(\cdot,z_*)K(x_i-x_{*})\,dz_*\Big)\cdot\Big(\nabla_{v_i}+\frac{\beta}{1+4\beta t}v_i\Big)F_{N,m}.
\end{multline*}
Note that the last term can be bounded by
\begin{eqnarray*}
\lefteqn{2\frac{N-m}N\sum_{i=1}^m\int_{\Xd^{m}} \omega_{m}\Big(\int_{\Xd} F_{N,m+1}(\cdot,z_*)K(x_i-x_{*})\,dz_*\Big)\cdot\Big(\nabla_{v_i}+\frac{\beta}{1+4\beta t}v_i\Big)F_{N,m}}\\
&\le&2\|K\|_{L^\infty}\sum_{i=1}^m\int_{\Xd^{m}}\omega_{m}\Big(\int_\Xd F_{N,m+1}(\cdot,z_*)\,dz_*\Big)\Big|\Big(\nabla_{v_i}+\frac{\beta}{1+4\beta t}v_i\Big)F_{N,m}\Big|\\
&=&2\|K\|_{L^\infty}\sum_{i=1}^m\int_{\Xd^{m}} \omega_{m}F_{N,m}\Big|\Big(\nabla_{v_i}+\frac{\beta}{1+4\beta t}v_i\Big)F_{N,m}\Big|.
\end{eqnarray*}
Further using Young's inequality and taking advantage of the dissipation, we are led to
\begin{equation*}
\partial_t\int_{\Xd^m} |F_{N,m}|^2\omega_{m}
\,\le\,Cm\int_{\Xd^{m}}|F_{N,m}|^2\omega_{m},
\end{equation*}
for some constant $C$ only depending on $d,\beta,\|K\|_{L^\infty}$.
Hence, by Gr\"onwall's inequality, we deduce for all $1\le m\le N$ and $t\ge0$,
\begin{equation*}
\int_{\Xd^m} |F_{N,m}(t)|^2\,\omega_{m}
\,\le\, \Big(e^{Ct}\int_{\Xd} |f_\circ|^2e^{\frac\beta2|v|^2}\Big)^m.
\end{equation*}
Also note that for the mean-field dynamics we have a similar estimate, for all $t\ge0$,
\begin{equation*}
\int_\Xd|f(t)|^2\omega_1
\,\le\,
e^{Ct}\int_\Xd|f_\circ|^2e^{\frac\beta2|v|^2}.
\end{equation*}
Combining these two bounds and recalling the definition~\eqref{eq:defin-HNm} of linear correlations, we obtain the following global exponential a priori estimates, for all $1\le m\le N$ and $t\ge0$,
\begin{equation}\label{eq:apriori-tildeANm}
A_{N,m}\,:=\,\Big(\int_{\Xd^m} |H_{N,m}|^2\,\omega_{m}\Big)^\frac12
\,\le\,\Big(Ce^{Ct}\int_\Xd|f_\circ|^2e^{\frac\beta2|v|^2}\Big)^\frac m2.
\end{equation}
Next, repeating the computations towards~\eqref{eq:ANm-diff} in the proof of Proposition~\ref{prop:HNm}, we similarly find in the present setting, for all $1\le m\le N$ and $t\ge0$,
\begin{equation*}
\left\{\begin{array}{l}
\partial_t A_{N,m}\,\le\,
Cm A_{N,m}
+Cm A_{N,m+1}
+C\frac{m^3}{N^2}( A_{N,m-1}+ A_{N,m-2}),\\
A_{N,m}|_{t=0}=\mathds1_{m=0},
\end{array}\right.
\end{equation*}
for some constant $C$ only depending on $d,\beta,\|K\|_{L^\infty},\int_\Xd|f_\circ|^2e^{\frac\beta2|v|^2}$.
Appealing to Lemma~\ref{lem:hier} for this hierarchy of differential inequalities, together with the exponential a priori bounds~\eqref{eq:apriori-tildeANm}, the conclusion follows.
\end{proof}

With the above suboptimal estimates at hand, we can now conclude the proof of Theorem~\ref{th:main-Linfty}: we use the hierarchy of equations for nonlinear correlations, cf.~Lemma~\ref{lem:hier-corr}, and recover the optimal estimates by a direct induction.
We emphasize that for $K\in L^\infty$ the hierarchy for nonlinear correlations is indeed perfectly usable, as already exploited in~\cite{Hess_Childs_2023}. However, as this is a nonlinear hierarchy, it is not easy to directly use it to deduce correlation estimates: in our approach, it is critical to start from subcritical estimates on linear correlations, and only use the nonlinear hierarchy to get improvements.

\begin{proof}[Proof of Theorem~\ref{th:main-Linfty}]
Let $\omega_m$ be the time-dependent weight in~\eqref{eq:def-tildeomem}, and set for abbreviation
\[B_{N,m}\,:=\,\Big(\int_{\Xd^m}|G_{N,m}|^2\omega_m\Big)^\frac12,\qquad1\le m\le N,\quad t\ge0.\]
Arguing similarly as in the proof of Corollary~\ref{cor:estim-GNm-0}, the estimates on linear correlations in Proposition~\ref{prop:HNm-re} allow to deduce the following on nonlinear correlations: for all $1\le m\le N$ and $t\ge0$,
\begin{equation}\label{eq:est-BNm-0}
B_{N,m}(t)\,\le\,(Ce^{Ct}m^2)^m(N^{-\frac m2}+\mathds1_{m=1}),
\end{equation}
where henceforth multiplicative constants only depend on $d,\beta,\|K\|_{L^\infty}$, and $\int_\Xd|f_\circ|^2e^{\frac\beta2|v|^2}$.
In order to improve on these subcritical estimates, we now appeal to the hierarchy of equations for nonlinear correlations, cf.\@ Lemma~\ref{lem:hier-corr}.
By a direct estimate on this nonlinear hierarchy, using $K\in L^\infty$, similarly as in the proof of Proposition~\ref{prop:HNm-re} above, we obtain for all $1\le m\le N$ and $t\ge0$,
\begin{multline*}
\partial_t B_{N,m}^2\,\lesssim\, mB_{N,m}^2
+mB_{N,m+1}^2
+m\bigg(\sum_{\ell=1}^{m-2}\binom{m-1}{\ell}B_{N,\ell+1} B_{N,m-\ell}\bigg)^2\\
+ \frac{m^3}{N^2}\bigg( \sum_{\ell=0}^{m-2}\binom{m-2}{\ell} B_{N,\ell+1}B_{N,m-\ell-1}\bigg)^2
+\frac{m^3}{N^2} \bigg(\sum_{\ell=0}^{m-2}\sum_{r=0}^{m-\ell-2}\binom{m-1}{\ell,r} B_{N,\ell+1} B_{N,r+1} B_{N,m-\ell-r-1} \bigg)^2,
\end{multline*}
and thus, after time integration, with $B_{N,m}|_{t=0}\le C\mathds1_{m=1}$,
\begin{multline}\label{eq:estim-BNm}
B_{N,m}^2(t)\,\le\,Ce^{Ct}\mathds1_{m=1}
+Cm\int_0^te^{Cm(t-s)}B_{N,m+1}^2
+Cm\int_0^te^{Cm(t-s)}\bigg(\sum_{\ell=1}^{m-2}\binom{m-1}{\ell}B_{N,\ell+1} B_{N,m-\ell}\bigg)^2\\
+\frac{Cm^3}{N^2}\int_0^te^{Cm(t-s)}\bigg( \sum_{\ell=0}^{m-2}\binom{m-2}{\ell} B_{N,\ell+1}B_{N,m-\ell-1}\bigg)^2\\
+\frac{Cm^3}{N^2}\int_0^te^{Cm(t-s)}\bigg(\sum_{\ell=0}^{m-2}\sum_{r=0}^{m-\ell-2}\binom{m-1}{\ell,r} B_{N,\ell+1} B_{N,r+1} B_{N,m-\ell-r-1} \bigg)^2.
\end{multline}
From here, we show by a direct induction argument that the subcritical estimates~\eqref{eq:est-BNm-0} can be improved to the following: for all $1\le m\le N$ and $t,\theta\ge0$,
\begin{equation}\label{eq:rec-r}
B_{N,m}(t)\,\le\,(C_0e^{C_0t})^{3^\theta m}(m+\theta)^{2(m+\theta)}\big(N^{-\frac12(m+\theta)}+N^{1-m}\big).
\end{equation}
Choosing $\theta\ge m-2$, this will imply the desired conclusion.
We argue by induction on $\theta$. For $\theta=0$, this estimate~\eqref{eq:rec-r} already follows from~\eqref{eq:est-BNm-0}. Next, assuming that~\eqref{eq:rec-r} holds for some $\theta\ge0$, and inserting it into~\eqref{eq:estim-BNm}, provided $C_0\ge2C$, we immediately deduce that~\eqref{eq:rec-r} further holds with~$\theta$ replaced by~$\theta+1$. This concludes the proof.
\end{proof}

\subsection*{Acknowledgements}
M. Duerinckx acknowledges financial support from the European Union (ERC, PASTIS, Grant Agreement n$^\circ$101075879).\footnote{{Views and opinions expressed are however those of the authors only and do not necessarily reflect those of the European Union or the European Research Council Executive Agency. Neither the European Union nor the granting authority can be held responsible for them.}} P.-E. Jabin was partially supported by NSF DMS Grants 2508570 and 2205694.

\bibliographystyle{plain}
\bibliography{biblio}

\end{document}